\documentclass[11pt, reqno]{amsart}

\usepackage{caption, amssymb,amsfonts,amsmath, graphicx, geometry}
\usepackage[nice]{nicefrac}
\usepackage{mathrsfs}
\usepackage[colorlinks=true,bookmarksnumbered=true, pdfauthor={Jacob, M\"orters}]{hyperref}

\newcommand{\eps}{\varepsilon}
\newcommand{\ssup}[1] {{\scriptscriptstyle{({#1}})}}
\newcommand{\one}{{\mathsf 1}}

\newcommand{\R}{\mathbb R}

\newcommand{\1}{\bf{1}}

\newcommand{\E}{\mathbb E}
\renewcommand{\P}{\mathbb P}

\newcommand*{\dbar}[1]{\bar{\bar{#1}}}

\renewcommand{\phi}{\varphi}

\newcommand{\N}{\mathbb N}

\renewcommand{\P}{\mathbb P}

\newcommand{\F}{\mathfrak{F}}

\newcommand{\Tex} {T_{\rm ext}}

\newcommand{\U}{\mathcal U}

\newcommand{\St}{\mathscr S}
\newcommand{\Co}{\mathscr C}

\newcommand{\heap}[2]  {\genfrac{}{}{0pt}{}{#1}{#2}}
\newcommand{\sfrac}[2] {\mbox{$\frac{#1}{#2}$}}

\newtheorem{theorem}{Theorem}
\newtheorem{definition}{Definition}

\newtheorem{lemma}{Lemma}
\newtheorem{proposition}{Proposition}
\newtheorem{remark}{Remark}

\begin{document}

\title[The contact process on evolving scale-free networks] 
{Metastability of the contact process\\ on fast evolving scale-free networks} 

\author[Emmanuel Jacob, Amitai Linker and Peter M\"orters]{Emmanuel Jacob, Amitai Linker and Peter M\"orters}

\maketitle




\begin{quote}
{\small {\bf Abstract:} }
We study the contact process in the regime of small infection rates 
on finite scale-free networks with stationary dynamics based on simultaneous updating of all connections of a vertex. We allow the update rates of individual vertices to increase with the  strength of a vertex, leading to a fast evolution of the network. We first develop an approach for inhomogeneous networks with general kernel and then focus on two canonical cases, 
the factor kernel and the pre\-ferential attachment kernel.  For these specific networks we identify and analyse four possible strategies how the infection can survive for a long time. We show that there is fast extinction of the infection when neither of the strategies is successful, otherwise there is slow extinction and the most successful strategy determines the asymptotics of the metastable density as the infection rate goes to zero.  We identify the domains in which these strategies dominate in terms of phase diagrams for the exponent describing the decay of the metastable density.
\end{quote}

\vspace{0.2cm}

{\footnotesize
\vspace{0.2cm}
\noindent\emph{MSc Classification:} Primary 05C82; Secondary 82C22.

}

\setcounter{tocdepth}{1} 
\tableofcontents

\pagebreak[3]

\section{Introduction}

The spread of disease, information or opinion on networks has been one of the most studied problems in mathematical network science over the past decade. There has been tremendous progress related to a variety of spreading processes and underlying network models. For the vast majority of these studies the network has been assumed to be fixed -- at least on the time scales of the processes running on the network. Real networks however undergo change and this change is often on a similar time scale as the spreading processes running on the networks. The problem of temporal variability of the networks, and how this variability can interfere with processes on the network has received little attention so far in the mathematical literature. The aim of this paper is therefore to investigate the possible effects of stationary dynamics of a network on the spread of an infection by offering an extensive case study based on the following basic assumptions:
\begin{itemize}
\item {\bf Scale-free network model.} We look at the class of sparse \emph{inhomogeneous random graphs}. For this class vertices are labelled by indices from $\{1,\ldots, N\}$  and edges exist independently with the probability of
an edge $\{i,j\}$ given as $\frac1N \,p(i/N, j/N) \wedge 1$ for a suitable kernel $p\colon (0,1]\times(0,1] \to (0,\infty)$. We focus on two universal types of kernel, which produce scale-free networks. The factor kernel reproduces the  asymptotic connection probabilities of most standard scale-free network models without significant correlations when the vertices are ordered by decreasing strength. This includes the Chung-Lu~\cite{CL06}, Norros-Reittu~\cite{NR06} and configuration models~\cite{MR95}. The preferential attachment kernel reproduces the connection  probabilities of various  preferential attachment models~\cite{BA99, DM11}, see~\cite{vdH17} for a recent survey of static network models.\\[-2mm]

\item {\bf Fast network evolution.} We focus on \emph{fast dynamics}, 
which arise, for example, as a rough approximation of migration effects in networks where links correspond to physical proximity. In our network dynamics all edges adjacent to a vertex are updated simultaneously at a rate that may depend increasingly on the vertex strength, so that the most relevant vertices can update relatively quickly. Upon updating a vertex loses all its connections and new connections are built independently. The connection probabilities of a vertex remain the same before and after the update, so that the network evolution is stationary. 
\\[-2mm]

\item {\bf SIS type epidemic process on the network.}
We investigate the \emph{contact process}, or SIS infection. The key feature which makes this process interesting from our point of view is that in order to survive the infection travels many times along individual edges, so that temporal changes in the status of  edges become relevant for the behaviour of the infection. For the contact process on scale-free networks with the factor kernel this feature leads to different qualitative behaviour of the static model and its classical mean-field approximation, as explained in~\cite{CD09}, see also~\cite{BB+, GMT05}. It is therefore a natural question to ask how the contact process behaves for dynamic models that interpolate between the static and the mean-field case. \\[-2mm]
\end{itemize}

\pagebreak[3]

\noindent
An \emph{evolving network} is a (random) family~\smash{$(\mathscr G^{\ssup N}_t \colon t\geq0, N\in\N)$}
of  graphs, where $\mathscr G^{\ssup N}_t$ has vertex set $\{1,\ldots,N\}$.  Conditionally on this network evolution, the \emph{contact process} on  
\smash{$(\mathscr G^{\ssup N}_t \colon t\geq0)$}  is a time-inhomogeneous Markov process that can be defined as follows: Every vertex $v$ may be healthy or infected; if infected, every adjacent healthy vertex gets infected with rate~$\lambda$ up until the recovery of $v$, which happens at rate one. When a vertex recovers it is again susceptible to infection. We write $X_t(v)=1$ if the vertex  $v\in\{1,\ldots,N\}$ is infected at time~$t$ and $X_t(v)=0$ otherwise. 
\medskip

The state when every vertex is healthy is absorbing and can be reached at any time from every other state in finite time with positive probability. Hence there exists a finite time $\Tex$, called the \emph{extinction time},  which is the infimum over all times where the contact process is in the absorbing state. 
%
If the evolving network is itself a (time-homogeneous) Markov process, then $(\mathscr G^{\ssup N}_t\!\!, X_t)_{t\ge 0}$ is also a (time-homogeneous) Markov process, and we will work within this context. More precisely, the evolving network we consider is a stationary Markov process, and unless otherwise specified, we start the process with the network distributed according to the stationary measure, and the contact process with every vertex infected. 
Our interest is in the size of the extinction time in that case.
\medskip

We say that the system experiences \emph{fast extinction} if, for some sufficiently small  infection rate $\lambda>0$, the expected extinction time is bounded by a power of the network size. We say that we have  \emph{slow extinction} if, for every infection rate $\lambda>0$, the expected extinction time 
is at least exponential in the network size with high probability, more precisely there exists a positive constant~$c$ such that, uniformly in $N>0$, we have
$$\mathbb{P}(\Tex \le e^{cN}) \le e^{-c N}.$$
 \pagebreak[3]
Slow extinction is a phenomenon of \emph{metastability}, a physical system reaching its equilibrium very slowly because it spends a lot of time in states which are local energy minima, the so-called
metastable states. Metastability in our model suggests, informally, that starting from all vertices infected the density of  infected vertices is likely to decrease rapidly to a \emph{metastable density}, and stay close to this density up to the exponential survival time of the infection. Metastable densities for the contact process
have been studied in the case of static networks by Mountford et al in~\cite{MVY13}. Our interest in metastable densities stems from the fact that, when seen as a function of small $\lambda$, they reflect which is the optimal survival strategy for the infection. As we shall see, the optimal survival strategies changes as we vary the network parameters, defining phase transitions.
\medskip

To understand the mechanisms behind slow extinction we follow~\cite{BB+} and first
look at a \emph{star graph}, ie a single central vertex connected to $k$ neighbouring vertices of degree one.  
If only the centre is initially infected, at the time of its first recovery it has on average $\frac{\lambda k}{\lambda+1}\sim \lambda k$ infected neighbours. The probabiliy that none of these neighbours reinfects
the centre is therefore approximately
$$\big(\sfrac1{1+\lambda}\big)^{\lambda k} \sim e^{-\lambda^2 k}.$$
Hence the infection survives for a long time on the star graph if $k \gg \lambda^{-2}$ 
and in this case the survival time is exponential  in~$\lambda^2k$. 
\pagebreak[3]
\medskip

If now the central vertex in the star graph updates at fixed rate~$\kappa$, and upon updating is connected to $k$ uninfected vertices, at the time~$r$  of the first recovery we have on average order~$\lambda k$ infected neighbours. The probability that none of them reinfects the centre before it updates is
$$\frac{\kappa}{\lambda^2k + \kappa}$$
and in this case we call $r$ a {\sl true recovery} (as opposed to {\sl simple recovery}) since then (and only then) the infection becomes extinct. Again the infection survives 
for a long time on the star graph if $k \gg \kappa \lambda^{-2}$ but now the survival time in this
case is of order~$\lambda^2k$, ie linear in the degree of the central vertex as
opposed to exponential as in the case of the star graph without updating.
\medskip

To understand the survival of infections on a (static or evolving)  inhomogeneous random graph we classify vertices as {\sl stars} and {\sl connectors} where stars have large degree and connectors do not. Assuming that the kernel~$p$ is decreasing in every component, we use a function $a(\lambda)\downarrow0$ to perform such a classification, where the set of stars $\St$ is
$$\St:=\big\{1,\ldots, \lfloor a(\lambda) N \rfloor \big\},$$ 
and its elements have degree asymptotically bounded from below by \smash{$\int_0^1 p(a(\lambda),x)\, dx$}. We think of stars acting locally like the centres in a star graph (hence the name) with most of their neighbours in the complementary set $\Co$ of connectors. 
In particular, an individual star can hold the infection
for a long time if $$\int_0^1 p\big(a(\lambda),x\big)\, dx \gg \lambda^{-2}.$$
Slow extinction of the infection is based on a collective strategy such that, given that a positive proportion of 
vertices in the  set $\St$  is infected, up to an exponentially small error probability a positive proportion 
of vertices in $\St$ will again be infected after a time span given by the recovery cycle of the stars. 
The existence of such a strategy ensures that the infection is kept alive on $\St$ for an exponentially long time, making this set the \emph{skeleton} of the infection. To obtain the metastable density associated to any given {\sl survival strategy} we find first a maximal function $a(\lambda)\downarrow 0$ (which defines $\St$) such that the strategy holds, and  obtain the density as the number of infected direct neighbours of $\St$ divided by the total number~$N$ of vertices.
\medskip

We have identified \emph{four relevant survival strategies} for the infection:\\[-3mm]

\begin{itemize}
\item[(i)] \emph{Quick direct spreading}\\[-4mm]

\noindent
Stars directly infect sufficiently many other stars before simple recoveries, so that the infection can be kept 
alive for a long time on the subgraph of stars alone. The connectors play no role for the survival of the infection.\\[-2mm]

\item[(ii)] \emph{Delayed direct spreading}\\[-4mm]

\noindent
As described for the star graph above, in this mechanism a star can retain an infection on a longer time scale if the lower bound on its degree is of larger order than 
$\lambda^{-2}$.  Operating on this longer time-scale stars spread the infection directly to other stars and keep the infection alive.\\[-2mm]\pagebreak[3]

\item[(iii)] \emph{Quick indirect spreading}\\[-4mm]

\noindent
Stars infect a large number of their neighbours before a simple recovery, and these neighbours then pass on the infection to other stars. In this way stars indirectly infect sufficiently many other stars keeping the infection alive. \\[-2mm]

\item[(iv)] \emph{Delayed indirect spreading}\\[-4mm]

\noindent As described for the delayed direct mechanism, a star retains the infection on a longer time scale if the lower bound on its degree is of larger order than 
$\lambda^{-2}$. On this time-scale stars pass the infection to other stars via their infected neighbours, as in the quick indirect mechanism. 
\end{itemize}
\pagebreak[3]
\medskip

Assume now that \smash{$\mathscr G^{\ssup N}_0$} is an inhomogeneous random graph with a kernel~$p$ 
and suppose that in the evolving network
\smash{$(\mathscr G^{\ssup N}_t \colon t\geq 0)$} every vertex updates with rate $\kappa$ 
and upon updating it receives a new set of adjacent edges with the same probability as before, 
given by the kernel~$p$. We now formulate and explain heuristically our results for the case of updating with constant rate~$\kappa$ in the case of the factor and preferential attachment kernel. Results for more general kernel and update rules will be formulated in the next section when we present our main results. 
\medskip

Define
$$I_N(t):= \frac1N \, \E\Big[\sum_{v=1}^N X_t(v) \Big],$$ 
to be the expected density of infected vertices at time $t$. Using the self-duality of the contact process~\cite[Chapter~VI]{L85} we get
\begin{equation}\label{dual}
I_N(t)= \frac1N \, \sum_{v=1}^N \P_v\big( \Tex>t \big),
\end{equation}
where $\P_v$ refers to the contact process started with only vertex~$v$ infected.
We say that the contact process
has \emph{lower metastable density} $\rho^-(\lambda)$ and \emph{upper metastable density} $\rho^+(\lambda)$ if, whenever $t_N$ is going to infinity slower than exponentially, we have\footnote{Actually, a more precise and slightly stronger metastability definition is given in next section.}
$$ 0<\rho^-(\lambda)=\liminf_{N \to \infty} I_N(t_N) \le \limsup_{N \to \infty} I_N(t_N)=\rho^+(\lambda).$$
We say that $\xi$ is the \emph{metastability exponent} of the process if the lower and upper metastability
densities exist for sufficiently small $\lambda>0$ and satisfy
$$\xi := \lim_{\lambda\downarrow 0} \frac{\log \rho^-(\lambda)}{\log \lambda}
= \lim_{\lambda\downarrow 0} \frac{\log \rho^+(\lambda)}{\log \lambda}.$$
Loosely speaking the metastability exponent measures the rate of decay of the metastable density as the infection rate~$\lambda$ approaches the critical value zero.
\medskip

We first look at the \emph{factor kernel}
$$p(x,y)=\beta x^{-\gamma} y^{-\gamma},
\qquad \mbox{ for some $\beta>0$ and $0<\gamma<1$. }$$
It is easy to see that the inhomogeneous networks with kernel~$p$ are 
scale free with \emph{power-law exponent} $\tau=1+\frac1\gamma$.
Our first result shows that in the case of factor kernels
there are two phase transitions in the behaviour of the contact process 
with small infection rates.
\medskip

\pagebreak[3]

\begin{proposition}\label{metastability_prop}
Suppose $p$ is a factor kernel with parameter $0<\gamma<1$. \\[-2mm]

\begin{itemize}
\item[(a)] If $0<\gamma<\frac13$ we have fast extinction, and if
$\frac13<\gamma<1$ we have slow extinction. \\[-2mm]
\item[(b)]
If $\frac13<\gamma<1$ the metastability exponent exists and equals
$$\xi=\left\{\begin{array}{lcl}{\frac 2 {3\gamma-1}}&\mbox{ if }&\frac13< \gamma<\frac 23 ,\\[1mm]
{\frac \gamma {2\gamma-1}}&\mbox{ if }&\frac23<\gamma<1\end{array}\right.$$
\end{itemize}
\end{proposition}
\smallskip

(a)~is the main result of Jacob and M\"orters~\cite{JM15}. 
\medskip

We argue now informally that in the regime $1/3< \gamma<2/3$ the strategy of delayed direct spreading prevails, whereas for $\gamma>2/3$ it is quick direct spreading that is most successful. For $\gamma<1/3$ none of the strategies succeed. 
\medskip

Under quick direct spreading the infection can be sustained on $\St$ if $a(\lambda)$ satisfies
$$\int_0^{a(\lambda)} \int_0^{a(\lambda)} \lambda p(x,y) \, dx \, dy \approx a(\lambda),$$
which arises from equating the initial amount of infected stars with the vertices in $\St$ infected by those stars before one unit time, which is the average time it takes to have simple recoveries. This equation yields \smash{$a(\lambda) \approx \lambda^{\frac{1}{2\gamma-1}}$} which is admissible if $\gamma>\frac12$. 
We hence get a lower bound for the lower metastable density
$$\rho^-(\lambda) 
\approx
\int_0^{a(\lambda)}\int_0^1  \lambda p(x,y) \, dx \, dy \approx  \lambda a(\lambda)^{1-\gamma}
\approx \lambda^{\frac\gamma{2\gamma-1}}.$$ 
For the delayed mechanism, on the other hand, we note that the lower bound on the expected degree of a star
is $a(\lambda)^{-\gamma}$ and hence the infection can be held at a star on a time scale of
$$T(\lambda)=\lambda^2 a(\lambda)^{-\gamma},$$
which the average time until a true recovery. Now by the same principle as in the quick mechanism $a(\lambda)$ has to satisfy
$$T(\lambda) \int_0^{a(\lambda)} \int_0^{a(\lambda)} \lambda p(x,y) \, dx \, dy \approx a(\lambda),$$
hence \smash{$a(\lambda) \approx \lambda^{\frac{3}{3\gamma-1}}$} which is admissible if $\gamma>\frac13$. This yields a lower bound of the form
$$\rho^-(\lambda) 
\approx \lambda a(\lambda)^{1-\gamma} \approx 
\lambda^{\frac2{3\gamma-1}}.$$ 
Comparing both densities, the delayed strategy therefore wins if $\frac13< \gamma< \frac23$, but if $\gamma>\frac23$ the quick strategy wins. The other two strategies we have identified turn out to be inferior in any case. If $\gamma<\frac13$ none of the strategies succeeds, ie gives an admissible value of $a(\lambda)$.
\bigskip

The situation is quite different for \emph{preferential attachment kernels} given by
$$p(x,y)=\beta (x\wedge y)^{-\gamma} (x \vee y)^{\gamma-1},
\qquad \mbox{ for some $\beta>0$ and $0<\gamma<1$. }$$
As before the networks with kernel~$p$ are easily seen to be
scale free with the same power-law exponent~\smash{$\tau=1+\frac1\gamma$.}
\medskip
\pagebreak[3]

\begin{proposition}\label{metastability_prop_2}
Suppose $p$ is a preferential attachment kernel with parameter $0<\gamma<1$. 
\begin{itemize}
\item[(a)] For all $0<\gamma<1$ there is slow extinction. \\[-2mm]
\item[(b)] The metastability exponent exists and equals
$$\xi=\left\{\begin{array}{lcl}\frac{3-2\gamma}{\gamma} &\mbox{ if }& \gamma< \frac35,\\[1mm]
{\frac{3-\gamma}{3\gamma-1}}&\mbox{ if }&\gamma> \frac35 .\end{array}\right.$$
\end{itemize}
\end{proposition}
\smallskip

We now explain heuristically that in the regime $\gamma<3/5$ the strategy of delayed direct spreading prevails, whereas for $\gamma>3/5$ it is delayed indirect spreading that is most successful. 
\medskip

For delayed direct spreading  $a(\lambda)$ again has to satisfy
$$T(\lambda) \int_0^{a(\lambda)} \int_0^{a(\lambda)} \lambda p(x,y) \, dx \, dy \approx a(\lambda),$$
for the time scale $T(\lambda)=\lambda^2 a(\lambda)^{-\gamma}.$
For the preferential attachment kernel this  gives \smash{$a(\lambda) \approx \lambda^{3/\gamma}$}, 
which is always admissible. This mechanism then yields
$$\rho^-(\lambda) 
\approx \lambda a(\lambda)^{1-\gamma} \approx 
\lambda^{\frac{3-2\gamma}{\gamma}}.$$ 
For the indirect mechanism the equation for $a(\lambda)$ changes to
$$T(\lambda) \int_0^{a(\lambda)} \!\!\int_{a(\lambda)}^1 \,\int_0^{a(\lambda)} \lambda^2 p(x,y)p(y,z) \, dx \, dy \, dz \approx a(\lambda),$$
where the term on the left represents the amount of stars infected by connectors that where in turn infected by the initially infected stars in a time-scale of order $T(\lambda)$. This gives \smash{$ a(\lambda) \approx\lambda^{\frac4{3\gamma-1}}$} for $\gamma>1/2$, which is admissible 
and yields 
$$\rho^-(\lambda) 
\approx \lambda a(\lambda)^{1-\gamma} \approx 
\lambda^{\frac{3-\gamma}{3\gamma-1}}.$$ 
Comparing once again the resulting densities, the indirect strategy therefore wins if $\gamma>3/5$, otherwise the direct strategy wins.
\pagebreak[3]\medskip

To better understand the metastability phenomenon and explore the full range of possible optimal strategies we move in the next section to a model where update rates can depend on the vertex strength. A rich and beautiful picture emerges from this.

\section{Statement of the main results}

Recall that for 
$N\in\N$ the inhomogeneous random graph~$\mathscr G^{\ssup N}$ 
has vertex set~$\{1,\ldots,N\}$ and every edge $\{i,j\}$ exists independently with probability
$$p_{i,j}:=\sfrac{1}{N} \, p\big(\sfrac{i}{N},\sfrac{j}{N}\big) \wedge 1,$$
where $p\colon (0,1]\times(0,1]\rightarrow (0,\infty)$ is a kernel for which we make the following assumptions:
\begin{enumerate}
\item $p$ is symmetric, continuous and decreasing in both parameters,
\item there is some $\gamma\in(0,1)$ and constants $0<c_1<c_2$ such that for all $a\in(0,1)$,
\begin{equation}\label{condp}
c_1 a^{-\gamma}\le p(a,1) \le \int_0^1 p(a,s) ds<c_2 a^{-\gamma}.
\end{equation}
\end{enumerate}
Observe that for every $f\colon(0,1] \to(0,\infty)$ decreasing, continuous and integrable, the kernel $p(s,t)= (s\wedge t)^{-\gamma} f(s\vee t)$ satisfies conditions (1) and (2). The choices $f(x)= \beta x^{-\gamma}$ and $f(x)=\beta x^{\gamma -1}$ give the factor and preferential attachment kernels, respectively.\medskip

We take $\mathscr G^{\ssup N}_0=\mathscr G^{\ssup N}$ and
obtain the evolving network~$(\mathscr G^{\ssup N}_t \colon t\geq 0)$ using the following dynamics:
Each vertex~$i$ updates independently with rate 
$$\kappa_i\;=\;\kappa_0\left(\frac{N}{i}\right)^{\gamma\eta}
\qquad \mbox{ for } i\in\{1,\dots,N\},$$
where $\eta\in\R$ and $\kappa_0>0$ are fixed constants. When vertex~$i$ updates, 
every unordered pair $\{i,j\}$, for $j\not= i$ forms an edge with probability $p_{i,j}$,
independently of its previous state and of all other edges. The remaining edges $\{k,l\}$ with $k, l\neq i$ remain unchanged. 
\medskip

Observe that this evolution is stationary. 
The expected degree of vertex~$i$ does not depend on time and is of order $(N/i)^{\gamma}$ so $\kappa_i$ is proportional to its degree raised to the power~$\eta$. If $\eta>0$ powerful vertices update more quickly and as $\eta$ passes from zero to $\infty$ we interpolate between the evolving networks with fixed update rates and the mean field model in which no memory of edges present is retained. We call this a \emph{fast evolving dynamics}. Conversely  if $\eta<0$ powerful vertices update slowly and we can consider the connection between them as fixed during long periods of time. As $\eta$ passes to $-\infty$ we interpolate between evolving networks with fixed update rates and the static model. In this work we  only consider the fast evolving case~$\eta\geq0$ 
as the slowly evolving case requires additional techniques.
\medskip

We say the contact process on the evolving graphs $(\mathscr G_t^{(N)})$ exhibits
\begin{itemize}
\item \emph{metastability} if there there exists $\eps>0$ such that
\begin{itemize}
\item whenever $t_N$ is going to infinity slower than $e^{\eps N}$, we have
$$ \liminf_{N\to\infty} I_N(t_N)>0.$$
\item whenever $s_N$ and $t_N$ are going to infinity slower than $e^{\eps N}$, we have 
$$I_N(s_N)-I_N(t_N) \underset{N\to \infty} \longrightarrow 0.$$
\end{itemize}
In that case, we can unambiguously define the \emph{lower metastable density} $\rho^-(\lambda)=\liminf I_N(t_N)>0$ and the \emph{upper metastable density} $\rho^+(\lambda)  =\limsup I_N(t_N)$.\\[-2mm]
\item \emph{a metastable density} $\rho(\lambda)$ if there is metastability and $\rho^-(\lambda)=\rho^+(\lambda)=\rho(\lambda)$. Equivalently, whenever $t_N$ is going to infinity slower than $e^{\eps N}$, we have
$$ \lim I_N(t_N)=\rho(\lambda)>0.$$
\end{itemize}

%

\pagebreak[3]

In the following theorem we  identify conditions on the kernel~$p$ for the four survival strategies identified in the first section to successfully sustain the infection. We deduce slow extinction and metastability and derive lower bounds on the lower metastable densities in each case. We also believe that there is a metastable density as soon as there is slow extinction, but we do not prove this.

\begin{theorem}\label{teolower} 
Define $\theta = \exp(- 2(1+\kappa_0 2^{\gamma \eta}))$. 
For $a,\lambda>0$ define $T=T(a,\lambda)$ by
\begin{equation}\label{T(a,lambda)}
T\log^2(T)=\sfrac{c_1 \theta}{20 \kappa_0^2}\, \lambda^2 a^{-\gamma(1-2\eta)}, 
\end{equation}
where $c_1$ is as in~\eqref{condp}. There exist positive and finite universal constants $M_{(i)}$, $M_{(ii)}$, $M_{(iii)}$, $M_{(iv)}$, such that slow extinction and metastability for all $\lambda \in (0,1)$ are guaranteed as soon as one can find $a=a(\lambda) \in (0,1/2)$ satisfying at least one of the following conditions:
\begin{itemize}
\item[(i)] {\bf (Quick Direct Spreading)}
$$\lambda ap(a,a)>M_{(i)}.$$

\item[(ii)] {\bf (Quick Indirect Spreading)}
\begin{equation}\;\lambda^2a p(a,1)^2>M_{(ii)}.\notag
\end{equation}

\item[(iii)] {\bf (Delayed Direct Spreading)} 
\begin{equation}T(a,\lambda)>M_{(iii)}\;\;\mbox{ and }\;\;\lambda a T(a,\lambda) \, p(a,a)>M_{(iii)}.\notag
\end{equation}

\item[(iv)] {\bf (Delayed Indirect Spreading)} 
\begin{equation}
T(a, \lambda) >M_{(iv)}\;\;\mbox{ and }\;\;\;\lambda^2 a T(a, \lambda) p(a,1)^2>M_{(iv)}.\notag
\end{equation}
\end{itemize}
Moreover, in each of these cases we have 
\begin{equation}
\label{lowdensityetapos}
\rho^-(\lambda)\;\geq\;c'(\lambda a p(a,1)\wedge 1),
\end{equation}
where $c'>0$ is a universal constant (independent of $\lambda$).
\end{theorem}

\bigskip
While the lower bounds above can be verified by investigating each of the four explicit survival strategies separately, upper bounds require a general, more implicit, method that yields information independent of any chosen strategy. Our approach is a supermartingale technique which gives upper bounds based on the choice of a scoring function. By carefully selecting a proper scoring function 
the technique will produce upper bounds which match the lower bounds in each of the cases investigated here. 

\bigskip

\begin{theorem}
\label{teoupper}
Fix $\lambda>0$ small and define the time-scale function $T_{\lambda}\colon (0,1)\rightarrow (0,\infty)$ as $$T_{\lambda}(x)=\max\big\{\lambda^{2}x^{-\gamma(1-2\eta)},1\big\}.$$
Fix $D>0$ such that
\begin{equation}
\label{condD}
D\;\leq\;\min\left\{\frac{\kappa_0}{4},\frac{\kappa_0^2(1-\gamma)}{64\beta},\frac{1}{16}\right\}.
\end{equation}
\begin{enumerate}
\item Suppose there is some function $S\colon (0,1]\rightarrow(0,\infty)$ with  $\int_0^1S(x)\, dx<\infty$ and
\begin{equation}
\label{defiT}
T_{\lambda}(x)\leq\mathfrak{c}S(x)^{\delta},
\end{equation}
for some $\mathfrak{c}>0$ and $0<\delta<1$. Suppose further that, for $N$ sufficiently large, 
\begin{equation}\label{condS1}
\lambda T_{\lambda}\left(\sfrac{x}{N}\right)\sum_{y=1}^{N} p_{x,y}\, S\left(\sfrac{y}{N}\right)\;\leq\;DS\left(\sfrac{x}{N}\right) \quad \mbox{ for all~$x\in\{1,\dots,N\}$.}
\end{equation}
Then, for sufficiently small $\lambda>0$, there is some $\omega'=\omega'(\lambda)$ such that, for 
large~$N$,
\begin{equation}\label{resextime}
\E\big[T_{\rm ext}\big]\;\leq\; \omega' N^{\delta}, \notag
\end{equation}
and in particular there is fast extinction.\\

\item If there exists some $a=a(\lambda)>0$ and some continuous function $S:[a,1]\rightarrow (0,\infty)$ such that, for all $N$ large,
\begin{equation}\label{condS2}
\lambda T_{\lambda}\left(\sfrac{x}{N}\right)\sum_{y=1}^{N} p_{x,y}\, S\left(\sfrac{y\vee \lceil aN\rceil}{N}\right)\;\leq\;DS\left(\sfrac{x}{N}\right) \mbox{ for all $x\in\{\lfloor aN\rfloor+1,\dots,N\}$, }
\end{equation}
then there exists $\omega>0$ 
and a function $\eps=\eps(N)$ converging to 0 as $N\uparrow\infty$ such that, for all $N$, $\delta>0$ 
and all $t\ge0$, we have
\begin{equation}
\label{resdens}
I_N(t)\;\leq\;a+\frac{1}{S(a)}\int_a^1 S(y)\, dy + \frac 1 {\delta \omega t} \int_a^1 S(y)^\delta\, dy+\eps(N).
\end{equation}
In particular, if there is metastability, then the upper metastable density satisfies
\begin{equation}
\label{upperdens}
 \rho^+(\lambda) \le a(\lambda)+\frac{1}{S(a(\lambda))}\int_{a(\lambda)}^1 S(y)\, dy.
\end{equation}
\end{enumerate}
\end{theorem}

\medskip
Applying these two theorems to the kernels considered yields our main result.

\begin{theorem}
\label{teofinal}\ \\[-4mm]
\begin{itemize}
\item[(a)]
Suppose $p$ is the factor kernel. 
\begin{itemize}
\item[(i)] If $\eta<\frac12$ and \smash{$\gamma<\frac{1}{3-2\eta}$}, or if $\eta\ge \frac12$ and \smash{$\gamma<\frac12$}, there is fast extinction.
\item[(ii)]
 if $\eta<\frac12$ and \smash{$\gamma>\frac{1}{3-2\eta}$}, or if $\eta\ge \frac12$ and \smash{$\gamma>\frac12$}, there is slow extinction and metastability, and the metastability exponent satisfies
\begin{equation}
\label{dens1}
\xi \:=\;\left\{\begin{array}{cc}\frac{2-2\gamma\eta}{3\gamma-2\gamma\eta-1}&\mbox{ if }\;\gamma<\frac{2}{3+2\eta},\\[-2mm]
&\\\frac{\gamma}{2\gamma-1}&\mbox{ if }\;\gamma>\frac{2}{3+2\eta}.\end{array}\right.
\end{equation}\end{itemize}
\medskip

\item[(b)] Suppose $p$ is the preferential attachment kernel.  
\begin{itemize}
\item[(i)] 
If $\eta\ge \frac12$ and \smash{$\gamma<\frac12$}, there is fast extinction.
\item[(ii)]
If $\eta<\frac12$, or if $\eta\ge \frac12$ and \smash{$\gamma>\frac12$}, there is slow extinction and metastability, and the metastability exponent satisfies
\begin{equation}
\label{dens2}
\xi\:=\;\left\{\begin{array}{ccl}\frac{3-2\gamma-2\gamma\eta}{\gamma-2\gamma\eta}&\mbox{ if }& \eta<\frac12\ \mbox{ and }\ 0<\gamma<\frac{3}{5+2\eta},\\[-2mm]&\\\frac{3-\gamma-2\gamma\eta}{3\gamma-2\gamma\eta -1}&\mbox{ if }&\eta<\frac12\ \mbox{ and }\ \frac{3}{5+2\eta}<\gamma<\frac{1}{1+2\eta},\\[-2mm]&\\ \frac{1}{2\gamma-1}&\mbox{ if }&\frac{1}{1+2\eta}<\gamma. \end{array}\right.
\end{equation}
\end{itemize}
\end{itemize}
\end{theorem}

\begin{figure}[h!]
    \centering
    {{\includegraphics[width=7.3cm]{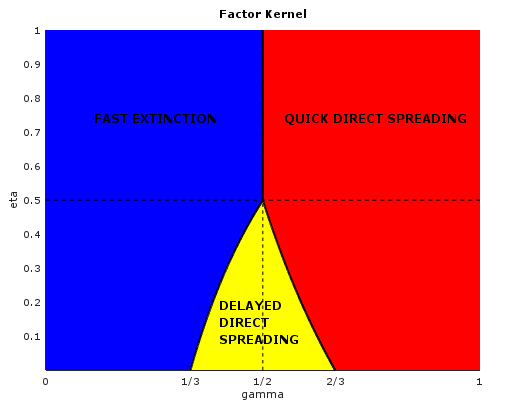} }}%
    \quad
    {{\includegraphics[width=7.3cm]{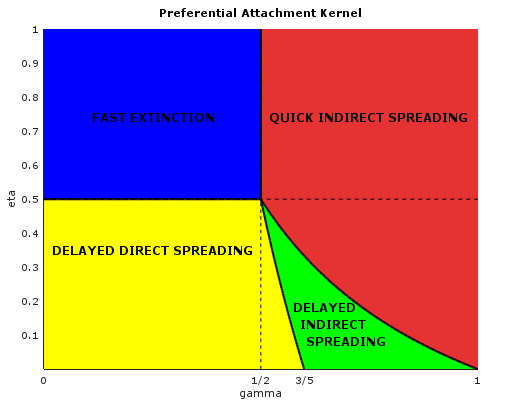} }}
\caption{The figures summarise Theorem~\ref{teofinal} in the form of phase diagrams for the factor kernel (left) and the preferential attachment kernel (right).}
\end{figure}

\begin{remark}
For both kernels, when $\eta>\frac12$ and $\gamma<\frac12$ we have 
$T_\lambda(x)\sim1$ and hence \eqref{defiT} holds for any $\delta>0$. 
By Theorem~\ref{teoupper}\,(1) we then get that $\E[T_{\rm ext}]$ is even 
subpolynomial in~$N$.
\end{remark}

\begin{remark}
The different exponents for the metastable densities in Theorem~\ref{teofinal} 
are indicative of different survival strategies for the infection, as  indicated in Figure~1.
Propositions~1 and~2 follow from Theorem~\ref{teofinal}  by letting $\eta=0$.
\end{remark}

\begin{remark}
In the cases of slow extinction, our results are actually slightly more precise than stated in Theorem~\ref{teofinal}. In particular, the upper metastable density always satisfies
\smash{$\rho^+(\lambda)\le c \lambda^{\xi}$} for some constant $c$. In the phases when quick direct/indirect spreading prevails, the lower metastable density also satisfies 
\smash{$\rho^-(\lambda)\ge c \lambda^{\xi}$} for some $c>0$, so in these phases we obtain the metastable densities up to a bounded multiplicative factor. 
\end{remark}



\bigskip
\pagebreak[3]

The rest of this paper is organized as follows; In Section~3 we introduce a graphical representation of the evolving network and contact process which allows us to define the process rigorously, yielding at the same time useful properties such as self-duality and monotonicity. In Sections~4 and~5 we give the proofs of Theorems~\ref{teolower} and \ref{teoupper} respectively, and finally in Sections~6 and~7 we apply those theorems to deduce Theorem~\ref{teofinal}. 


\section{Graphical representation}

The evolving network model $({\mathscr G}^{\ssup N}_t \colon t\geq 0, N\in \N)$ is represented 
with the help of the following independent random variables;
\begin{enumerate}
\item[(1)] For each $x\in\N$, a Poisson point process \smash{$\mathcal U^x=(U^x_n)_{n\ge 1}$} of intensity $\kappa_x$, describing the updating times of the vertex~$x$. Given $x\not=y$ we also write 
\smash{$\mathcal U^{x,y}=(U^{x,y}_n)_{n\ge 1}$} for the union $\mathcal U^x \cup \mathcal U^y$, 
which is a Poisson point process of intensity $\kappa_x+\kappa_y$, describing the updating times of the potential edge $\{x,y\}$. \smallskip
\item[(2)] For each $\{x,y\}$ with $x\not=y$ and $x,y\leq N$, a sequence of independent random variables $(C^{x,y}_n)_{n\ge0}$, all Bernoulli with parameter $p_{x,y}$, 
describing the presence/absence of the edge in the network after the successive updating times of the potential edge $\{x,y\}$. More precisely, if $t\ge 0$ then 
$\{x,y\}$ is an edge in~\smash{${\mathscr G}^{\ssup N}_t$} if and only if  $C^{x,y}_n=1$ for $n= \vert[0,t]\cap \U^{x,y}\vert$. We denote ${\mathcal C}^x:=(C^{x,y}_n\colon y\leq N, n\in\N)$.\smallskip
\end{enumerate}
Given the network we represent the infection by means of the following set of
independent random variables;
\begin{enumerate}
\item[(3)] For each $x\in \N$, a Poisson point process $\mathcal R^x=(R^x_n)_{n\ge 1}$ of intensity one describing the recovery times of~$x$.\smallskip
\item[(4)] For each $\{x,y\}$ with $x\not=y$, a Poisson point process $\mathcal {\mathcal I}_0^{x,y}$ with intensity $\lambda$ describing the infection  times along the edge $\{x,y\}$. Only the trace $\mathcal I^{x,y}$ of this process on the~set
$$\bigcup_{n=0}^\infty \{[U^{x,y}_n, U^{x,y}_{n+1}) \colon  C^{x,y}_n=1\} \subset [0,\infty)$$
can actually cause infections. Write $(I^{x,y}_n)_{n\ge1}$ for the ordered points of $\mathcal I^{x,y}$.
If just before time $I^{x,y}_n$ vertex $x$ is infected and $y$ is healthy, then $x$ infects $y$ at time  $I^{x,y}_n$. If $y$ is infected and $x$
healthy, then $y$ infects $x$.  Otherwise, nothing happens.\smallskip
\end{enumerate}
The infection is now described by a process $(X_t(x), x\in \{1,\ldots,N\}  \colon t\geq 0)$ with values 
in~\smash{$\{0,1\}^N$}, such that $X_t(x)=1$ if $x$ is 
infected at time~$t$, and $X_t(x)=0$ if $x$ is healthy at time $t$.  
More formally, the infection process 
associated to this 
graphical representation and to a starting set $A_0$ of infected vertices, is the c\`adl\`ag process with $X_0(x)=\one_{A_0}(x)$ 
evolving only at times $t\in \mathcal R^x\cup \bigcup_{n=1}^\infty I_n^{x,y}$, according to the following rules:
\begin{itemize}
\item If $t\in \mathcal R^x$, then $X_t(x)=0$ (whatever $X_{t-}(x)$). 
\item If $t  \in \mathcal I^{x,y}$, then 
$$
(X_t(x),X_t(y))=
\left\{
  \begin{array}{rl}
    (0,0) & \mbox{ if } (X_{t-}(x),X_{t-}(y))=(0,0). \\
    (1,1) & \mbox{ otherwise.} \\
  \end{array}
\right.
$$
\end{itemize}
The process $({\mathscr G}^{\ssup N}_t, X_t \colon t\geq 0)$ is a Markov process describing the simultaneous evolution of the network and of the infection. We call $(\F_t\colon t\ge 0)$ its canonical filtration.


Using the graphical representation we obtain monotonicity and duality properties of the contact
process on the evolving graph. The proof of the following proposition
is standard within the context of the contact process, see~\cite{L85}, and therefore omitted here.

\begin{proposition}\label{dualprop}\ \\[-3mm]
\begin{enumerate}
\item {\bf Monotonicity.} 
If $(X^1_t \colon t\geq 0)$, $(X^2_t\colon t\geq 0)$ are processes constructed as above with $X^1_0\leq X^2_0$ and infection rates $\lambda_1\leq \lambda_2$, then $X^1_t\leq X^2_t$ stochastically.\smallskip
\item {\bf Self-Duality.} If $X^A, X^B$ correspond to the process with initial condition $X_0={\bf 1}_{A}$ and $X_0={\bf 1}_{B}$ respectively, then for all $t>0$,
$$\P\big(\exists x\in A,\;X^B_t(x)=1\big)\;=\;\P\big(\exists x\in B,\;X^A_t(x)=1\big).$$
\end{enumerate}
\end{proposition}
The only added subtlety in the proof of the proposition above when compared to~\cite{L85} is that the duality property combines both the duality property of the contact process, and that of the network dynamics.

\section{Slow extinction and lower bounds}

In this section we prove Theorem~\ref{teolower} by showing four different {\sl survival strategies} which can sustain the infection exponentially long. All these strategies are based on a division between powerful and weak vertices given by a parameter $a=a(\lambda)\in (0,1/2)$ 
as
$$\St\;:=\;\{1,2,\ldots,\lfloor aN\rfloor\},\;\;\;\;\;\Co\;:=\;\{\lfloor aN\rfloor+1,\ldots, N\}.$$
The elements of $\St$ are called \emph{stars} and the elements of $\Co$ are called \emph{connectors}.
Notice that when $\lambda$ decreases, any vertex with fixed degree has a lower chance of infecting its neighbours. Our definition of a star changes accordingly, that is, $a(\lambda)\downarrow0$ as $\lambda\downarrow0$. We denote by $\St_0=\{x\in \St, X_0(x)=1\}$ the set of initially infected stars, and by $S_0=|\St_0|$ its cardinality.
We start the proof with a relatively simple lemma, which already contains the flavour of the kind of inequalities we will use throughout the proof.
\medskip

\begin{lemma}\label{lemmalower}
Fix $r>0$ and suppose one is given an initial condition $(X_0,\mathscr{G}_0)$ such that $S_0\ge r a N$. Then there exists a constant $C>0$ (independent of $\lambda$, $a$, $N$) such that, for all~$t\in[2,3]$,
\begin{equation}
\label{lemma1r}
 \E\Big[\sum_{v=1}^N X_t(v) \Big| X_0,\mathscr{G}_0 \Big] \;\geq\; C (\lambda a p(a,1) \wedge 1) r N.
 \end{equation}
\end{lemma}

\begin{remark}
Lemma~\ref{lemmalower} remains true if $t$ is in an arbitrary compact set bounded away from zero, changing only the value of $C$. For our purposes the above formulation 
suffices.
\end{remark}
\begin{proof}
We introduce a terminology specific to this proof as follows:
\begin{eqnarray*}
\St'& :=&\{x\in\St \colon \mathcal{R}^x\cap[0,2]=\emptyset, \mathcal{U}^x\cap[0,1]\neq\emptyset\}, \\
\Co'&:=&\{y\in\Co \colon \mathcal{R}^y\cap[1,3]=\emptyset, \exists x \in \St', \; \mathcal{I}^{x,y}\cap[1,2]\neq\emptyset\}. 
\end{eqnarray*}
Each $x\in \St_0$ belongs to $\St'$ independently with probability $e^{-2}(1-e^{-\kappa_x}) \ge e^{-2}(1-e^{-\kappa_0})$. Therefore the cardinality $S'$ of $\St'$ dominates a binomial random variable with parameters $S_0$ and $e^{-2}(1-e^{-\kappa_0})$.

For $y\in \Co$, the event $\mathcal R^y\cap[1,3]=\emptyset$ has probability $e^{-2}$ and is independent of the event $E^y:=\{\exists x\in  \St', \; \mathcal{I}^{x,y}\cap[1,2]\neq\emptyset\}$, whose probability we want to estimate. Conditionally on $x\in \St'$, a sufficient condition for the event $E^y$ to be satisfied is that
\begin{enumerate}
\item $\mathcal{I}_0^{x,y}\cap[1,2]\neq\emptyset$, which happens with probability $1-e^{-\lambda}$.
\item At time $t= \min \left( \mathcal{I}_0^{x,y}\cap[1,2]\right)$, the edge $\{x,y\}$ belongs to the network. This happens with probability $p_{x,y}$, independently of the configuration of the network at time 0, as vertex $x$ has updated on the time interval $[0,1]$.
\end{enumerate}
The probabilities here obtained are independent from each other and from the realization chosen for the $\mathcal{U}^x,\mathcal{R}^x$, hence $x$ infects $y$ on time interval $[1,2]$ (namely $\mathcal I^{x,y}\cap[1,2]\ne \emptyset$) with probability at least \smash{$p_{x,y}(1-e^{-\lambda})$}, which we can bound below by \smash{$(1-e^{-\lambda}) p(a,1)/N$}, by the definition of $p_{x,y}$ and the monotonicity of $p(\cdot, \cdot)$. If we now condition on $\St'$, the events $\mathcal I^{x,y}\cap[1,2]\ne \emptyset$ are independent and we get
\begin{eqnarray*}
\P(E^y \big| \St') &\ge& 1-\left(1- (1-e^{-\lambda}) \sfrac {p(a,1)}N\right)^{S'}\\
&\ge& 1- \exp\big(- (1-e^{-\lambda}) \sfrac {p(a,1) S'}N\big) \ge \sfrac {\lambda p(a,1) S'}{4N} \wedge \sfrac 12,
\end{eqnarray*}
where in the last inequality we used twice the inequality $1-e^{-x}\ge (x\wedge 1)/2$ for $x\ge 0$ (we also used $\lambda<1$). Finally we obtain that, given the initial condition of the network and the infection and conditionally on $\St'$, the cardinality of $\Co'$ dominates a binomial random variable with parameters $|\Co| \ge N/2$ and $$\rho=\sfrac {e^{-2}\lambda p(a,1) S'}{4N} \wedge \sfrac {e^{-2}}2.$$
To conclude, we first consider the case $\lambda a p(a,1)<1$. Then we always have $\rho=\sfrac {e^{-2}\lambda p(a,1) S'}{4N}$, and we easily get that the expectation of $|\Co'|$ is bounded from below by
$$ \frac {e^{-3}(1-e^{-\kappa_0})} 8 \lambda a p(a,1) r N.$$
In the case $\lambda a p(a,1)>1$, we have $\rho\ge \sfrac {e^{-2} S'}{4 a N}$, and we get a bound of
$$ \frac {e^{-3}(1-e^{-\kappa_0})} 8 r N.$$
This altogether proves~\eqref{lemma1r} with $C= e^{-3}(1-e^{-\kappa_0})/8$.
\end{proof}
\medskip



Denoting by $\St_k$ the set of infected stars at time $k\in\N$, and by $S_k=|\St_k|$ its cardinality, we aim to prove that for some $r>0$ the events ${\mathscr E}^r_k:=\{S_k>raN\}$ hold for a sufficiently long time. With this in mind we say that a family of events ${\mathscr E}_k$ depending on $k\in\N$
holds \emph{exponentially long} if there exists $c>0$ such that, for all $N$, 
\[ 
\mathbb{P}\Big(\bigcap_{k\le e^{c N}}{\mathscr E}_k \Big) \ge 1- e^{-c N}. \\[-1mm]
\]
As ${\mathscr E}^r_k\subset\{T_{\rm ext}>k\}$, if the events ${\mathscr E}^r_k$ hold exponentially long, we have slow extinction. 
Moreover, from Lemma~\ref{lemmalower} we have
$$C\lambda a p(a,1)\P\left(|\{x\in \St \colon X_{\lfloor t\rfloor-1}(x)=1\}|>raN\right)\;\leq\; I_N(t)
$$
for any $t>2$, so if ${\mathscr E}^r_k$ holds exponentially long, the left hand side above is bounded from below by \smash{$C\lambda a p(a,1)(1- e^{-c N})$} for some $c>0$, and hence we deduce the lower bound~\eqref{lowdensityetapos} on the metastable density. Our aim is therefore not only to prove slow extinction, but the stronger result that under any of the conditions given in Theorem~\ref{teolower}, the events ${\mathscr E}^r_k$ hold exponentially long. Actually, we will also allow a conditioning on any intial configuration included in the event ${\mathscr E}^r_0$, and still show that the events ${\mathscr E}^r_k$ hold exponentially long.
\medskip

The proof of metastability, on the other hand, will also follow from this result. First, note that 
from \eqref{dual} ot is easy to see that $I_{N}(\cdot)$ is decreasing, and hence it suffices to show that, whenever $t_N\le e^{\eps N}$, 
$$ \limsup_{N\to \infty} | I_N(t)-I_N(t_N) | \underset {t\to \infty}\longrightarrow 0.$$
Using self-duality, we can write
$$I_N(t)-I_N(t_N)= \frac 1 N \,  \sum_{v=1}^N \P_v\big(t< \Tex<t_N\big),$$
where we recall that $\P_v$ stands for the probability measure corresponding to the infection starting with only vertex $v$ infected. So metastability follows if we can prove that $\P_v(t<\Tex< e^{\eps N})$ converges to 0 as $t\to \infty$, uniformly in $N$ and $v\in \{1, \ldots , N\}$. We separate this proof into three steps:
\begin{itemize}
\item[(1)] For every $n\ge0$, uniformly in $N$ and $v$,
$\displaystyle\P_v(\max S_k< n,\Tex >t) \underset {t\to \infty} \longrightarrow 0.$\\
\item[(2)] Uniformly in $N$ and $v$,
$\displaystyle \P_v(\max S_k\ge r a N |\max S_k \ge n) \underset {n\to \infty} \longrightarrow 1.$\\
\item[(3)] Uniformly in $v$,
$\displaystyle \P_v(\Tex\ge e^{\eps N} | \max S_k \ge r a N) \underset {N\to \infty} \longrightarrow 1.$
\end{itemize}
These three steps easily give the result. Indeed, for any given $\eps>0$, choosing $N_0$ as in (3), then $n$ as in (2), then $t$ as in (1), we obtain that for any $N\ge N_0$ and $v\in\{1,\ldots, N\}$,
\[
\P_v(t<\Tex< e^{\eps N})\le 3\eps.
\]

The \emph{first step} is fairly easy and relies on the observation that there is some constant $c>0$ depending only on $n$ such that for all $k\ge 0$,
\[
\P_v(S_{k+1}\ge n | \F_k) \ge c \1_{\{S_k>0\}},
\]
We deduce $\P_v(\max S_k<n , S_t>0) \le (1-c)^{\lfloor t \rfloor}$, which proves the first step.
\medskip

For the \emph{second step}, we introduce the stopping time $K:=\inf\{k\ge 0, S_k \ge n\}$ and prove $\P_v(\max S_k\ge r a N | K<+\infty, \F_K) \to \infty$ uniformly on $N$, $v$ and on the $\sigma-$field $\F_K$. In other words, we provide a uniform bound in all the possible configurations 
\smash{$({\mathscr G}^{\ssup{N}}_K,X_K)$} for which $S_K\ge n$. This bound follows from the analysis of the process $(S_k)$ below (see in particular Lemma~3.(1) and Lemma~4.(1)), and concludes the second step. 
\medskip

Finally, the \emph{third step} also follows from the analysis of the process $(S_k)$. If the stopping time $\tilde K:=\inf\{k\ge 0, S_k \ge r a N\}$ is finite, then the event $\mathscr E^r_{\tilde K}$ holds, and further the events $\mathscr E^r_{k+\tilde K}$ hold exponentially long, proving the third step.

\medskip

We now provide the detailed analysis of the process $(S_k)$. A useful tool, which we will use repeatedly, is the following large deviation estimate, which can be 
directly derived from Chernoff's inequality.

\begin{lemma} \label{ldev}
Let $X$ be a binomial random variable with parameters $n$ and $q$.  Then
$$
\P(X<snq)\;\leq\;e^{-Dn}\;\;\;\mbox{ where }D=sq\log(s)+(1-sq)\log\left(\sfrac{1-sq}{1-q}\right)
\mbox{ and } s<1.
$$ 
\end{lemma}


\subsection{Quick Direct Spreading}
Let us start with quick direct spreading, which is arguably the simplest mechanism, that makes no use of connectors and is only based on stars infecting directly other stars before recovery. This strategy can only succeed when the subgraph~$\St$ is sufficiently connected. In this case, for $x,y\in\St$ and small $\lambda>0$, our choice $\eta\geq 0$ implies that there is typically an updating event $\mathcal{U}^{x,y}$ between two infections in $\mathcal{I}_{0}^{x,y}$ so the times $\mathcal{I}^{x,y}$ when infections pass the edge $\{x,y\}$ can therefore be approximated by a Poisson point process with rate $\lambda p_{x,y}$. If $\lambda$ is small and $N$ large, these rates tend to zero and hence, during a brief interval of time, the infection starting from a single vertex is unlikely to infect twice the same vertex, resulting in the infection spreading like a Galton-Watson process.\smallskip

We use, specifically in this quick direct spreading subsection, the terminology
$$ \St'_k:= \{x\in \St_k \colon\mathcal{R}^x\cap[k,k+1]=\emptyset\},$$
and $S'_k=|\St'_k|$. Clearly, $\St'_k\subset \St_{k+1}$ and thus $S'_{k}\le S_{k+1}$, as infected stars that do not recover on $[k, k+1]$ are still infected at time $k+1$. Further, we let 
$$
\begin{aligned}
 \St''_k:= \{x\in \St \backslash\St_k \colon\;& \mathcal{U}^x\cap[k,k+1/2]\neq\emptyset, \mathcal{R}^x\cap[k+1/2,k+1]=\emptyset,  \\
 & \exists y\in \St'_k, \mathcal{I}^{x,y}\cap[k+1/2,k+1]\neq\emptyset\},
 \end{aligned}
 $$
and $S''_k=|\St''_k|$. Clearly, we also have $\St''_k\subset \St_{k+1}$, as the stars in $\St''_k$ have been infected on $[k+1/2, k+1]$ and did not recover on that time interval.

An advantageous property of $\St'_k$ and $\St''_k$, compared to $\St_{k+1}$, is that their conditional laws knowing $(\mathscr G^{\ssup N}_k, X_k)$, depend only on $\St_k$, and not on the network structure $\mathscr G^{\ssup N}_k$. 
So, the cardinality $S'_k$ of $\St'_k$ is (conditionally) a binomial random variable with parameters $S_k$ and $e^{-1}$.

Now, if $x$ is in $\St\backslash \St_k$, then it satisfies $\mathcal{U}^x\cap[k,k+1/2]\neq\emptyset$ and $\mathcal{R}^x\cap[k+1/2,k+1]=\emptyset$ with probability $e^{-1/2}(1-e^{-\kappa_x/2})\ge e^{-1/2}(1-e^{-\kappa_0/2}).$ 
Conditioning on this event and on $\St'_k$, a similar argument to the one used in the proof of Lemma \ref{lemmalower} gives that $x$ belongs to $S_k''$ with probability at least
\begin{eqnarray*}
 \left[1- \left(1-\sfrac {\lambda p(a,a)}{4N}\right)^{S'_k}\right] &\ge& \left[1- \exp \left(-\sfrac {\lambda p(a,a) S'_k}{4N}\right)\right] \\
&\ge& \sfrac {\lambda p(a,a) S'_k}{8N} \wedge \sfrac 12.
\end{eqnarray*}
As a consequence, $S''_k$ dominates a binomial random variable with parameters $S-S_k$ and 
$$e^{-1/2}(1-e^{-\kappa_0/2}) \left(\sfrac {\lambda p(a,a) S'_k}{8N} \wedge \sfrac 12 \right).$$ 
Gathering these results with $S_{k+1}\ge S'_k+S''_k$, we obtain a stochastic lower bound for the conditional distribution of $S_{k+1}$ given $\F_k$, which we exploit in the following lemma.
\begin{lemma}\label{discreteMC}
Suppose $\rho$, $\rho'$ and $c$ are three positive constants such that $\rho \rho'>1$. Suppose $M_0, M'_0, M_1, M'_1, \ldots$ is a process on $\{0, 1, \ldots, n\}$, adapted to a filtration $(\F_0,\F'_0,\F_1,\F'_1,\ldots)$, such that 
\begin{itemize}
\item given $\F_k$ the random variable $M'_k$  is binomially distributed with parameters $M_k$~and~$\rho$;
\item given $\F'_k$ the random variable $M_{k+1}-M'_k$ dominates a binomially distributed random variable
with parameters  $n - M_k$ and \smash{$\rho' \frac {M'_k} {n+1} \wedge c$}.
\end{itemize}
 Then there exist positive constants $r,l,\eps>0$ such that for large $n$:
\begin{enumerate}
\item For every initial condition $M_0=m_0$, the probability that the process $M_k$ goes above value $rn$ is at least $1- e^{-l m_0}$.
\item For every initial condition $M_0=m_0\ge rn,$ with probability at least $1- e^{-\eps n}$, the process $(M_k)$ stays above value $rn$ at all times $k\le e^{\eps n}$.
\end{enumerate} 
\end{lemma}

Under the hypothesis $\lambda a p(a,a)> {8e}/{e^{-1/2}(1-e^{-\kappa_0/2})},$ we can apply Lemma~\ref{discreteMC} with the choice $M_k=S_k$ $M'_k=S'_k$, 
$\F'_k=\sigma(\F_k,S'_k)$, $n=\lfloor aN \rfloor$, $\rho= e^{-1}$, 
$$\rho'= \sfrac{e^{-1/2}(1-e^{-\kappa_0/2})} 8 \lambda a p(a,a) \quad\mbox{ and }\quad c=\sfrac{e^{-1/2}(1-e^{-\kappa_0/2})} 2.$$ 
Item~(2) then completes the proof of slow extinction, while items~(1) and (2) complete the proof of metastability. Thus, using Lemma~\ref{discreteMC}, we have proven the quick direct spreading part of Theorem~\ref{teolower}, with \smash{$M_{(i)}={8e}/{e^{-1/2}(1-e^{-\kappa_0/2})}$}. 

\begin{proof}[Proof of Lemma~\ref{discreteMC}]
We first prove the second item. Choose \smash{$r<c \wedge \sfrac 1 {\rho'}$}, which, together with $\rho \rho'>1$, implies $r<\rho$. If $M_k\ge r' n$ with  $r'>r/\rho$, then Lemma~\ref{ldev} implies $M_{k+1}\ge M'_k \ge rn$ with probability at least $1- e^{-c n}$, for some constant $c>0$.
If $rn\le M_k \le r'n$, choosing $\iota$ in the nonempty interval 
\smash{$(\sfrac 1 {\rho'}, \rho \wedge \sfrac{c}{r \rho'})$}, a further application of Lemma~\ref{ldev} yields $M'_k\ge \iota M_k$, with error probability bounded by some $e^{-c n}$ (possibly with a new value of $c>0$). 

Further, we can bound $\rho' \frac {M'_k} {n+1}$ from below by  
$\iota \rho' \frac n {n+1} r$, which is in $(r,c)$ for large $n$. A last application of Lemma~\ref{ldev} yields $M_{k+1}-M'_k\ge  r (n- M_k)$ and then $M_{k+1}\ge r n$, with error probability bounded by $e^{-\eps n}$. This easily gives item~(2) of the Theorem (dividing $\eps$ by 2 if needed).
\medskip

For the first item we let $K:= \inf\{k\ge 0, M_k=0 \text{ or } M_k\ge r' n\},$ where $r'$ is a constant to be determined later. We possibly have $r'<r$, but~(2) still holds if we replace\footnote{Also, the reader can check that once the process has gone above level $r' n$, it is actually likely to go above level $r n$, too.} $r$ by $r\wedge r'$.
Under the hypothesis $r'<\eta/\rho'$ and $k< K$, we can bound below the law of $M_{k+1}$ knowing $M'_k$ by a binomial random variable with parameters $(1-r')n$ and $\rho' M'_k /(n+1)$. We now prove that for some well-chosen $l>0$, the process $(e^{-l M_{k\wedge K}})_{k\ge 0}$ is a positive supermartingale. Note that the result then follows from a standard stopping theorem.
\medskip


It suffices to prove the inequality
$$\E[e^{-l M_1} | \F_0] \le e^{-l M_0}$$
on the event $M_0<r' n$. But on this event, the Laplace transform of Binomial random variables easily gives the following:
\begin{eqnarray*}
\E[e^{-l M_1} | \F'_0]&\leq\left(1- \sfrac {\rho' M'_0}{n+1} (1-e^{-l})\right)^{(1-r') n} 
\le \exp\left(- \sfrac {(1-r')n}{n+1} \rho' M'_0 (1-e^{-l})\right).
\end{eqnarray*}
Thus, a further Laplace transform gives
\begin{eqnarray*}
\E[e^{-l M_1} | \F_0]&\le&\Big(1- \rho \big(1-e^{- \frac{(1-r')n}{n+1} \rho' M'_0 (1-e^{-l})}\big)\Big)^{M_0}\\
&\le& \exp\Big(-\rho M_0  \big(1-e^{- \frac{(1-r')n}{n+1} \rho' (1-e^{-l})}\big)\Big).
\end{eqnarray*}
When $l$ goes to 0, the last expression is \smash{$\exp(-l \frac {(1-r')n}{n+1} \rho \rho' M_0 (1+o(1)))$. }
Choosing $r'>0$ small and $n$ large so that \smash{$\frac {(1-r')n}{n+1} \rho \rho'>1$}, and then $l>0$ small, we can guarantee
\smash{$\E[e^{-l M_1} | \F_0] \le e^{-l M_0},$}
and this completes the proof.
\end{proof}

\pagebreak[3]

\subsection{Quick Indirect Spreading}

Quick indirect spreading is a mechanism similar to quick direct spreading in the sense that stars spread the infection before seeing a simple recovery event, but in this case the infection spreads to connectors which in turn infect stars again. Being a two stage mechanism, quick indirect spreading can be less efficient than its direct version as we see in the case of the factor kernel. However, as it relies on the connectedness of the whole network rather than the connectedness among stars it can be advantageous when the latter is scarce. As with the direct version, the quick update of stars allows us to think of valid infections as Poisson point processes with parameter $\lambda p_{x,y}$ and then approximate the behaviour of $X_t$ as a Galton-Watson process.
\medskip

We keep from the quick direct spreading subsection the notation $S_k=|\St_k|
$ as well as \smash{$S'_k=|\St'_k|=|\{x \in \St_k: \mathcal R^x\cap [k,k+1]= \emptyset \}|$}. Again, conditionally on the network and infection evolution up to time $k$, the cardinality $S'_k$ of $\St'_k$ is binomial with parameters $S_k$ and~$e^{-1}$. 
In order to consider indirect spreading, we now introduce the following terminology, specific to this subsection:
\begin{eqnarray*}
\Co_k&:=& \{y \in \Co: \mathcal U^y\cap[k,k+1/3] \ne \emptyset, \mathcal R^y\cap [k+1/3,k+1]= \emptyset, \\
&& \hspace{52mm} \exists x\in \St'_k, \mathcal I^{x,y} \cap [k+1/3,k+2/3] \ne \emptyset\}, \\
\St''_k&:=& \{x \in \St\backslash \St_k: \mathcal R^x\cap[k+2/3, k+1] = \emptyset, \exists y \in \Co_k, \mathcal I^{x,y}\cap [k+2/3,k+1]\ne \emptyset\}.
\end{eqnarray*}
We also denote the cardinality of these sets by $C_k=|\Co_k|$, and $S''_k=|\St''_k|$, respectively. It should be clear that $S_{k+1}\ge S'_k +S''_k$. 
Indeed, for each star $x\in \St''_k$, we can find a star $z\in \St'_k$ that infects a connector $y\in \Co_k$ on time interval $[k+1/3, k+2/3]$, which stays infected until it infects $x$ on time interval $[k+2/3, k+1]$. Note that the condition
that connectors in $\Co_k$ should update on $[k, k+1/3]$, is useful for the law of $C_k$, conditionally on $\St'_k$ and on the network and infection evolution up to time $k$, to actually only depend on $\St'_k$. 
More precisely, each $y\in \Co$ belongs to $\Co_k$ with probability at least 
$$ e^{-2/3} (1-e^{-\kappa_0/3})\left(1-(1-e^{-\lambda/3})\sfrac {p(a,1)}N\right)^{S'_k}\ge  e^{-2/3} (1-e^{-\kappa_0/3})\left(\sfrac {\lambda p(a,1) S'_k}{6N}\wedge \sfrac 12\right).$$
Hence $C_k$ dominates a binomial random variable with parameters $(1-a)N$ and $e^{-2/3} (1-e^{-\kappa_0/3})({\lambda p(a,1) S'_k}/(6N)\wedge \frac 12).$
Similarly, conditionally on $S_k$, $S'_k$ and $C_k$, we can bound the random variable
$S''_k$ from below
by a binomial random variable with parameters $aN-S_k$ and 
$e^{-1/3} ({\lambda p(a,1) C_k}/(6N)\wedge \frac 12).$
Lemma~\ref{discreteMC} has to be replaced by the following lemma.
\begin{lemma}\label{discreteMC2}
Suppose $\rho, \rho', \rho'', \rho'''$ and $c, c'$ are positive constants such that $\rho \rho' \rho'' \rho'''>1$, and $M_0, M'_0, M''_0, M_1, M'_1,M''_1, \ldots$ is a process on $\{0, 1, \ldots,n\}$ adapted to the  filtration $(\F_0, \F'_0,\F''_0, \F_1, \F'_1, \F''_1, \ldots)$ such that 
\begin{itemize}
\item given $\F_k$ the random variable $M'_k$  is binomially distributed with parameters $M_k$~and~$\rho$;
\item given $\F'_k$ the random variable $M''_k$ is dominating a  binomially distributed random variable with parameters $\lceil\rho''' n\rceil$ 
and \smash{$\sfrac {\rho'} {n+1} M'_k \wedge c$;}
\item given $\F''_k$ the random variable $M_{k+1}-M'_k$ dominates a binomially distributed random variable with parameters $n - M_k$, and \smash{$\sfrac{\rho''}{n+1} M''_k  \wedge c'$.}
\end{itemize}
 Then 
there exist positive constants $r,l,\eps>0$ such that for large $n$:
\begin{enumerate}
\item[(1)] For every initial condition $M_0=m_0$, the probability that the process $M_k$ goes above value $rn$ is at least $1- e^{-l m_0}$.
\item[(2)] For every initial condition $M_0=m_0\ge rn,$ with probability at least $1- e^{-\eps n}$, the process $(M_k)$ stays above value $rn$ at all times $k\le e^{\eps n}$.
\end{enumerate} 
\end{lemma}
The proof is similar to that of Lemma~\ref{discreteMC}. It just involves more calculation, which is not so informative, so we omit it. We can now apply this lemma with the parameters 
$$M_k=S_k, M'_k=S'_k, M''_k=C_k, \F'_k= \sigma(\F_k, S'_k), \F''_k=\sigma(\F'_k, C_k), n=\lfloor a N \rfloor,$$ 
$$\rho= e^{-1}, \rho'=e^{-2/3} (1-e^{-\kappa_0/3})\lambda a p(a,1)/12, \rho''=e^{-1/3} \lambda a p(a,1)/12,$$ and $\rho'''=(1-a)/a$, under the condition 
$\rho \rho' \rho'' \rho'''>1$, which is satisfied if $a<1/2$ and
\smash{$\lambda^2 a p(a,1)^2 > 288 e^2/(1-e^{-\kappa_0/3})$.} We now conclude the quick indirect spreading part of Theorem~\ref{teolower} just like the quick direct spreading part, with $M_{(ii)}=288 e^2/(1-e^{-\kappa_0/3})$.

\subsection{Delayed Direct Spreading}

\bigskip
Delayed direct spreading is a mechanism similar to quick direct spreading in the sense that the infection spreads directly from star to star. The main difference is that the infection is kept alive at a star on a longer 
time scale with the aid of connectors. A single vertex, if powerful enough, can survive a recovery event by infecting a connector which in turn infects it back before an updating event (where the connection is lost with a high probability) thus prologing the recovery cycle of stars. In contrast with the stars studied at \cite{BB+} which survive for an amount of time exponential in their degree, the survival time here is roughly linear in this parameter, which is explained by the cost of maintaining the right conditions on the network for this effect to take place.\smallskip

To begin our proof, for each $k\in\N$, $k\geq 1$ define $\bar{\Co}_k$ as
$$\bar{\Co}_k:=\{y\in\Co \colon [\mathcal{U}^{y}\cup\mathcal{R}^y]\cap[k,k+2]=\emptyset\}$$ 
that is, $\bar{\Co}_k$ is the set of all {\bf stable connectors} in the interval $[k,k+2]$. As each $y\in\Co,\;y>N/2$ belongs to \smash{$\bar{\Co}_k$} independently with probability at least $\theta=\exp(-2(\kappa_0 2^{\gamma\eta}+1))$, Lemma~\ref{ldev} shows that $\P(|\bar{\Co}_k|>\theta N/4)>1-e^{-cN}$ for some fixed $c>0$ and hence these events hold exponentially long. \emph{For the entire remainder of this section we therefore fix a realization of $\mathcal{U}^y,\mathcal{R}^y, y\in\Co$ such that 
$\{|\bar{\Co}_k|>\theta N/4\}$ holds exponentially long and all probabilities will be taken to be conditional on such a realization. }\smallskip

Next we 
denote by $\St_k$ the set of infected stars at time $kT$, and $S_k=|\St_k|$. Note that we use the same notation as before, though the length of a recovery cycle has been modified, from one to $T=T(a,\lambda)$. Our hope is that this will not confuse the reader, but rather stress the unity of the approach.



\subsubsection{Properties of stars}

As the probability $p_{x,y}$ of having a connection between a star~$x$ and a connector $y$ is bounded from below by \smash{$\sfrac1N {p(a,1)} \ge \sfrac1N {c_1 a^{-\gamma}}$}, we deduce that for given $t$, the number of connectors \smash{$y\in \bar{\mathscr C_{\lfloor t\rfloor}}$} connected to $x$ dominates a binomial random variable with parameters $\lceil \theta N/4\rceil$ and $c_1 a^{-\gamma}/N$. By Lemma~\ref{ldev}, one can deduce that
\begin{equation}
\label{lotconnectors}
\P\big(\big|\{y\in\bar{\Co}_{\lfloor t\rfloor},\;\{x,y\}\in \mathscr G^{\ssup N}_t\}\big|>c_1 \theta a^{-\gamma}/5\big)>1-e^{-c a^{-\gamma}},
\end{equation}
for some $c>0$, uniformly in $a\le 1/2$ and $N\geq c_1 a^{-\gamma}$. In particular, any star $x$ maintains sufficiently many neighbouring stable connectors for an exponentially long time.
\begin{definition}
\label{Tstable}
For any $T>0$ let
$$\bar{\mathcal{U}}^x_T\; :=\;\{0\}\cup\bigcup_{U_m^x\in\mathcal{U}^x\cap[0,T]}\big\{U_m^x+n\kappa_x^{-1} \colon n\in\N_0\cap[0,\kappa_x(U_{m+1}^x-U_m^x)]\big\}.$$
A star $x\in\St$ is $T$-{\bf stable} if 
\begin{itemize}
\item[(i)] $|\bar{\mathcal{U}}^x_T|<3\kappa_x T$ and, 
\item[(ii)] at all times $t\in\bar{\mathcal{U}}^x_T$, the vertex $x$ has at least $\frac{c_1 \theta}{5} a^{-\gamma}$ neighbours in $\bar{\Co}_{\lfloor t\rfloor}$.
\end{itemize}
\end{definition}

The set $\bar{\mathcal{U}}^x_T$ arises by adding points to $\mathcal{U}^x$ between consecutive updating events when these are further than $\kappa_x^{-1}$ units of time apart. Loosely speaking,
$T$-stability means that~$x$ does not update too much and that upon updating it has enough neighbouring stable connectors. The next result follows directly from \eqref{lotconnectors} and a large deviation argument for $|\bar{\mathcal{U}}^x_T|$, we omit its easy proof.
\begin{lemma}
\label{Tstableprop}
$\displaystyle\lim_{T\rightarrow\infty}\liminf_{N\rightarrow\infty} \inf_{x\in\St} \P(x\mbox{ is }T\mbox{-stable})=1.$
\end{lemma}
\medskip

Note that the convergence when $T\to\infty$ is uniform in the possible values of $a< 1/2$ and $\lambda<1$. 
Since the convergence is uniform and the events $\{x$ is $T$-stable$\}_{x\in\St}$ are independent, we deduce that, when $T$ and $N$ are large, most stars will exhibit this property. We define next the concept of {\bf $[L,T]$-susceptibility} of~$x$, depending only on $\mathcal{U}^x, \mathcal{R}^x$, which 
loosely speaking means that recoveries of~$x$ are not too frequent, and not too close to its 
updating 
events. 

\begin{definition}
\label{suscept}
For any $L\in[0,T)$ we say that a star $x\in\St$ is {\bf $[L,T]$-susceptible} if
\medskip
\begin{enumerate}
\item[(i)] there are no recovery events in $[L,L+\kappa_x^{-1}]$ or $[T-\kappa_x^{-1},T]$,
\item[(ii)] $|\mathcal{R}^x\cap[L,T]|<2T$,
\item[(iii)] for every pair of consecutive updating times $L\leq t_1<t_2\in\bar{\mathcal{U}}^x_T$, such that\\
\noindent $\mathcal{R}^x\cap[t_1,t_2]\neq\emptyset$ we have
\begin{equation}
\label{espac}
(r_1-t_1)(t_2-r_2)>[\kappa_x^{2}T\log(T)]^{-1},
\end{equation}
where $r_1$ and $r_2$ are the first and last recoveries in $[t_1,t_2]$, respectively.
\end{enumerate}
\end{definition}

There is no reason most stars should be $[L,T]$-susceptible, however the probability of a $T$-stable $x\in\St$ having this property is bounded away from zero, if $T$ and $N$ are large.

\begin{lemma}
\label{LTsusceptibility}
There exists $q_1>0$ such that
$$\liminf_{T\rightarrow \infty}\liminf_{N\rightarrow\infty} \inf_{x\in\St}
\P\left(x\mbox{ is }[L,T]\mbox{-susceptible}\;\big|\; \mathcal{U}^x, \mathcal{C}^x\right)>q_1,$$
for all $L\in[0,T)$ and any realization of $\mathcal{U}^x, \mathcal{C}^x$ such that $x$ is $T$-stable.
\end{lemma}

\begin{proof}
Fix a realization of $\mathcal{U}^x, \mathcal{C}^x$ such that $x$ is $T$-stable and notice that in this case $[L,T]$-susceptibility depends on the process $\mathcal{R}^x$ alone. In fact, all three conditions in 
Definition~\ref{suscept} are decreasing with $\mathcal{R}^x$. By Harris' inequality the three conditions are therefore positively correlated and hence $\P\left(x\mbox{ is }[L,T]\mbox{-susceptible}\;\big|\;\bar{\mathcal{U}}^x_T\right)$ is larger than the product of the individual probabilities, which we now calculate. It is easy to see that conditions~(i) and~(ii)  are 
independent of $\mathcal{U}^x, \mathcal{C}^x$ and their probability is bounded from below by $e^{-2}$ and $1/2$ respectively (the latter value is obtained from a rough Markov inequality).
For the bound on (iii) we take two consecutive \smash{$t_i,t_{i+1}\in \bar{\mathcal{U}}^x_T$} with \smash{$\mathcal{R}^x\cap[t_i,t_{i+1}]\neq\emptyset$} and find an upper bound for the event that \eqref{espac} fails at 
\smash{$[t_i,t_{i+1}]$}. We split this event in two scenarios:

\medskip

\emph{First} suppose $r_1-t_i>1/(2\kappa_x)$ 
so for \eqref{espac} to fail we necessarily have $t_{i+1}-r_2\leq 2/(\kappa_x T\log T)$. This way we can bound the probability from above by
$$1-e^{-2/(\kappa_x T\log(T))}\;\leq\;\frac{2}{\kappa_x T\log(T)}\;\leq\;\frac{2}{\kappa_x T}.$$
By symmetry the treatment for the case $t_{i+1}-r_2>1/(2\kappa_x)$ is the same.

\emph{Second} suppose $r_1-t_i\leq 1/(2\kappa_x)$ and $t_{i+1}-r_2\leq 1/(2\kappa_x)$. In this case, the random variables $r_1-t_i$ and $t_{i+1}-r_2$ are dominated by independent exponential random variables. We deduce that the probability of this scenario is bounded by
$$\int_0^{\frac{1}{2\kappa_x}}\int_0^{\infty}e^{-y-z}{\bf 1}_{\{z<[y\kappa_x^{2}T\log(T)]^{-1}\wedge[2\kappa_x]^{-1}\}} \, dz \, dy.$$
Dividing the integral as to which expression is smaller in the indicator function yields
$$\int_0^{\frac{2}{\kappa_x T\log(T)}}\int_0^{\frac{1}{2\kappa_x}}e^{-y-z}\, dz \, dy\;+\;\int_{\frac{2}{\kappa_x T\log(T)}}^{\frac{1}{2\kappa_x}}\int_0^{\frac{1}{y\kappa_x^{2}T\log(T)}}e^{-y-z}\, dz \, dy.$$
The first term is smaller than $1/(\kappa_x^{2}T\log(T))\leq 1/(\kappa_x T)$ while the second is equal to
$$\int_{\frac{2}{\kappa_x T\log(T)}}^{\frac{1}{2\kappa_x}}e^{-y}\big(1-e^{-\frac{1}{y\kappa_x^{2}T\log(T)}}\big) \, dy\;\leq\;\int_{\frac{2}{\kappa_x T\log(T)}}^{\frac{1}{2\kappa_x}}\frac{1}{y\kappa_x^{2}T\log(T)}\, dy\;\leq\;\frac{2}{\kappa_x T}.$$
We have deduced that the probability that \eqref{espac} fails in $[t_i,t_{i+1}]$ is bounded from above by $3/(\kappa_x T)$. We thus obtain a lower bound for the probability that (iii) holds of the form
$$\big(1-\sfrac{3}{\kappa_x T}\big)^{|\bar{\mathcal{U}}^x_T\cap[L,T]|}\;\geq\;\big(1-\sfrac{3}{\kappa_x T}\big)^{3\kappa_x T}\;\geq\;e^{-10},$$
where the first inequality holds by $T$-stability of $x$, and the second holds if $\kappa_x T$ is large. 
\end{proof}

\begin{definition}
A star $x\in\St$ is {\bf $T$-infectious} if
$$\big|\big\{t\in\mathcal{U}^x \cap [0,T]
\colon(\mathcal{U}^x\cup\mathcal{R}^x)\cap[t,t+\kappa_x^{-1})=\{t\}\big\}\big|>\sfrac12 {e^{-1-\kappa_0^{-1}}}\, \kappa_xT.$$
\end{definition}
This property does not imply that the star actually infects other vertices but rather that upon survival on $[0,T]$ it is infected for a sufficiently large proportion of the time, hence giving it a larger chance of doing so. A large deviation argument 
yields that
$\liminf_{T\to \infty}\liminf_{N\rightarrow\infty}\P(x\mbox{ is }T\mbox{-infectious})=1,$
so most of the stars will have this property. Gathering all the results obtained here we obtain the following lemma.

\begin{lemma}
\label{starprop}
There exists $q_1>0$ as in Lemma~\ref{LTsusceptibility} such that, for all large $T$ and $L\in[0,T)$,
$$\liminf_{N\rightarrow\infty} \inf_{x\in\St} \P(x\mbox{ is }T\mbox{-infectious, }T\mbox{-stable and }[L,T]\mbox{-susceptible})>q_1.$$
\end{lemma}

\medskip

\subsubsection{Survival and spreading}
The star properties mentioned in the last lemma are useful to bound from below the probability that a star maintains the infection on $[0,T]$ and infects another star.
\begin{definition}
A star $x\in\St$ is {\bf $[L,T]$-infected} if $X_L(x)=X_T(x)=1$ and for all $t\in \bar{\mathcal{U}}^x_T\cap[L,T],\;X_t(x)=1$.
\end{definition}
If $x$ is $[0,T]$-infected, we say the infection is maintained at $x$ on $[0,T]$ (although the star may of course have recovered several times on this time interval). For the next lemma we take $T=T(a,\lambda)$ defined as in Theorem~\ref{teolower}.
\begin{lemma}
\label{longsurvival}
There exists $q_2>0$ independent of $N$, such that for all $\lambda$ and $a$ making $T=T(a, \lambda)$ sufficiently large, and for all $x\in\St$ and $L\in[0,T)$, 
\begin{equation}
\label{survupdate}
\P\big(x \text{ is } [L,T]\text{-infected}\;\big|\;X_L(x)=1,\;x \text{ is } T\text{-stable and }[L,T]\text{-susceptible}\big)>q_2.
\end{equation}
\end{lemma}

\begin{proof} Fix a realization of $\mathcal{U}^x,\mathcal{R}^x$ and $(C_n^{x,y})_{y\in\Co}$ 
making $x$ a $T$-stable and $[L,T]$-susceptible star, and assume that $X_L(x)=1$. Let $t_j$ be the first event in $\bar{\mathcal{U}}^x_T$ after $L$, which by definition lies in $[L,L+\kappa_x^{-1}]$; from the first condition of $[L,T]$-susceptibility there are no recoveries in this interval and hence $X_{t_j}(x)=1$. Since $X_{t_j}(x)=1$, using conditional probabilities and the Markov property we can bound from below the probability in the statement by
\begin{equation}
\label{proba1}
\prod_{\substack{t_i\in\bar{\mathcal{U}}^x_T\cap[L,T]\\\mathcal{R}^x\cap[t_{i},t_{i+1}]\neq\emptyset}}\P(X_{t_{i+1}}(x)=1\;\big|\;X_{t_i}(x)=1),
\end{equation}
where the restriction in the product follows since we trivially have $\P(X_{t_{i+1}}(x)=1\;\big|\;X_{t_i}(x)=1)=1$ when there are no recoveries in $[t_{i},t_{i+1}]$. Now, in the case $\mathcal{R}^x\cap[t_{i},t_{i+1}]\neq\emptyset$, we define $r_1$ and $r_2$ as the first and last element in that intersection, respectively, and notice that a sufficient scenario for $X_{t_{i+1}}(x)=1$ is that
\begin{itemize}
\item $x$ infects some neighbour $y\in\bar{\Co}_{\lfloor t_i\rfloor}$ during $[t_i,r_1]$,
\item since $y\in \bar{\Co}_{\lfloor t_i\rfloor}$, it remains infected (and also a neighbour of $x$) up until time $t_{i+1}$,
\item $y$ infects $x$ back during $[r_2,t_{i+1}]$. 
\end{itemize}
The scenario above follows from the event
\begin{equation}
\label{event1}
\bigcup_{\heap{x\sim y \text{ at time }t_i}{y\in\bar{\Co}_{\lfloor t_i\rfloor}}} \Big\{\mathcal{I}_0^{x,y}\cap[t_i,r_1]\neq\emptyset\;\;\;\mbox{ and }\;\;\;\mathcal{I}_0^{x,y}\cap[r_2,t_{i+1}]\neq\emptyset\Big\},
\end{equation}
and using independence we can calculate its probability as
$$1-\prod_{\heap{x\sim y \text{ at time }t_i}{y\in\bar{\Co}_{\lfloor t_i\rfloor}}}
\big[1-\big(1-e^{-\lambda(r_1-t_i)}\big)\big(1-e^{-\lambda(t_{i+1}-r_2)}\big)\big].$$
Since $t_{i+1}-t_{i}<\kappa_x^{-1}$ the exponents in the expression above tend to zero and hence we can use the bound
$$(1-e^{-\lambda(r_1-t_i)})(1-e^{-\lambda(t_{i+1}-r_2)})\;\geq\;\frac{\lambda^{2}(r_1-t_i)(t_{i+1}-r_2)}{4}\;\geq\;\frac{\lambda^{2}}{4\kappa_x^{2}T\log T},$$
which follows from the third condition of $[L,T]$-susceptibility, to obtain that 
\eqref{event1} has probability at least
$$1-\prod_{\heap{x\sim y \text{ at time }t_i}{y\in\bar{\Co}_{\lfloor t_i\rfloor}}}
\big[1-\sfrac{\lambda^{2}}{4\kappa_x^{2}T\log(T)}\big]\;\geq\;1-\exp\left(- \frac{\lambda^{2}c_1 \theta a^{-\gamma}}{20 a^{-2\gamma\eta}\kappa^2_0 T\log(T)}\right)\;=\;1-\frac{1}{T},$$
where the inequality follows from the $T$-stability of $x$ and the monotonicity of $\kappa$ 
and the equality from the definition of $T$. This allows us to bound the argument in \eqref{proba1} by $(1-1/T)$, and from $[L,T]$-stability there are at most $2T$ recoveries, giving
$$|\{t_i\in\bar{\mathcal{U}}^x_T\cap[L,T]\colon
\mathcal{R}^x\cap[t_{i},t_{i+1}]\neq\emptyset\}|\;\leq\;2T,$$
so we finally obtain
$$\prod_{\substack{t_i\in\bar{\mathcal{U}}^x_T\cap[L,T]\\\mathcal{R}^x\cap[t_{i},t_{i+1}]\neq\emptyset}}\P(X_{t_{i+1}}(x)=1\;\big|\;X_{t_i}(x)=1)\;\geq\;\left(1-\sfrac{1}{T}\right)^{2T}\;\geq\;e^{-4},$$
which holds if $T$ is large, giving the result.
\end{proof}

Denote by $\St'_0 \subset \St_0$ the set of initially infected stars, that are also $T$-infectious, $T$-stable, $[0,T]$-susceptible, and $[0, T]$-infected. From Lemmas~\ref{starprop} and~\ref{longsurvival}, when $T$ and $N$ are large, conditionally on $S_0$, the random variable $S'_0=|\St'_0|$ dominates a binomial random variable with parameters $S_0$ and $q_1 q_2$.

Further, we bound from below the number of initially uninfected stars, that get infected by some star in $\St'_0$, and are still infected at time $T$.
\begin{lemma}
\label{spreadingStDD}
We have, uniformly in $\lambda<1$ and $a<1/2$,
\begin{align}\label{spreadingstDD}
\liminf_{T\to\infty}\liminf_{N\rightarrow\infty} \inf_{x\in\St \backslash \St_0} \P\big(\exists y\in{\St}_0',\;  t\in {\mathcal I}^{x,y}\cap[0,T]\;& \colon X_t(y)=1 \big| \St'_0 \big) 
> \sfrac{\lambda T p(a,a) S'_k}{8e^{1+1/\kappa_0} N} \wedge \sfrac 12.\,\,
\end{align}
\end{lemma}

\begin{proof} Fix $y\in{\St}_0'$ and consider $t_i\in\mathcal{U}^y\cap[0,T]$ such that $[t_i,t_{i}+\kappa_x^{-1}]\cap(\mathcal{R}^y\cup\mathcal{U}^y)=\emptyset$, which exists since $y$ is $T$-infectious. As \smash{$y\in{\St}_0'$}, it satisfies the condition in Lemma~\ref{longsurvival} and hence it is infected throughout the interval~$[t_i,t_{i}+\kappa_x^{-1}]$, which gives a small time interval for $y$ to infect $x$. To find a lower bound for the event $\{{\mathcal I}^{x,y}\cap[t_{i},t_{i}+\kappa_x^{-1}]\neq\emptyset\}$ it is enough that ${\mathcal I}_0^{x,y}\cap[t_{i},t_{i}+\kappa_x^{-1}]\neq\emptyset$ and that at the first infection event the edge $\{x,y\}$ belongs to the graph, but this happens with probability
$$p_{x,y}\big[1-e^{-\lambda\kappa_x^{-1}}\big]\;\geq\;1-e^{-\frac{\lambda p_{x,y}}{2\kappa_x}}.$$
Using independence we deduce
$$\P(\exists t\in \mathcal{I}^{x,y}\cap[0,T] \colon X_t(y)=1) \geq1-\!\!\!\prod_{[t_i,t_{i}+\kappa_x^{-1}]\cap\mathcal{R}^y=\emptyset}\exp\left(-\sfrac{\lambda p_{x,y}}{2\kappa_x}\right)\;\geq\;1-\exp\left(-\sfrac{\lambda p_{x,y} T}{4e^{1+1/\kappa_0}}\right),$$
where the last inequality is due to the fact that $y\in{\St}_0'$ and hence it is $T$-infectious. Finally, to deduce \eqref{spreadingstDD}, we use independence one last time to deduce
$$\begin{array}{rcl}\P\left(\exists y\in{\St}_0',\;t\in \mathcal{I}^{x,y}\cap[0,T],\;X_t(y)=1\right)&\geq&1-\exp\Big(-\frac{\lambda T}{4e^{1+1/\kappa_0}}\displaystyle\sum_{y\in{\St'}_0}p_{x,y}\Big)\\
& \geq& 1-\exp\left(-\sfrac{\lambda T p(a,a) S'_k}{4e^{1+1/\kappa_0} N} \right)
\\ & >& \sfrac{\lambda T p(a,a) S'_k}{8e^{1+1/\kappa_0} N} \wedge \sfrac 12.
\end{array}$$
\end{proof}

Using Lemma~\ref{spreadingStDD} we get that, when $T$ and $N$ are large, each $x\in\St\backslash{\St}_0$ has probability at least \smash{$\sfrac{\lambda T p(a,a) S'_k}{8e^{1+1/\kappa_0} N} \wedge \sfrac 12$} to receive an infection in $[0,T]$. Calling $T_x\in[0,T]$ the first time when this occurs, we can use Lemma~\ref{LTsusceptibility} to deduce that with probability at least $q_1$ the star $x$ is $T$-stable and $[T_x,T]$-susceptible. Now, being infected at time $T_x$, Lemma~\ref{longsurvival} gives that with probability at least $q_2$, the star $x$ will be infected at time $T$. Since all these events are independent for different values of $x$, we deduce that $S_1-S'_0$ dominates a 
binomial random variable with parameters $\lfloor a N \rfloor - S_0$ and 
$$q_1 q_2\left(\sfrac { \lambda T p(a,a) S'_k}{8e^{1+1/\kappa_0} N} \wedge \sfrac 12\right).$$

Finally, the same reasoning applies for the whole process $S_k$, which can now be studied similarly as for quick direct spreading. More precisely, if $T$ is large and under the hypothesis
$$ \sfrac {q^2_1 q^2_2 \lambda T a p(a,a) }{8e^{1+1/\kappa_0}}>1,$$
we can apply Lemma~\ref{discreteMC} with $n=\lfloor aN \rfloor$, $\rho=q_1 q_2$, $\rho'=\sfrac {q_1 q_2 \lambda T a p(a,a) }{8e^{1+1/\kappa_0}}.$

In that case, slow extinction and metastability follow just as before. To actually deduce the lower bound for the lower metastable density, Inequality~\eqref{lowdensityetapos}, we need not only that the events $S_k\ge r a N$ hold exponentially long for some $r>0$ (as we get from Lemma~\ref{discreteMC}), but that the events $|\{x\in\St: X_k(x)=1\}|\ge r' a N$ with $k\in\N$, hold exponentially long, for some $r'>0$. However it is clear from our proofs that we do get this, for some $r'<r$. 
This altogether proves the case of delayed direct spreading of Theorem~\ref{teolower}.
\medskip

\subsection{Delayed Indirect Spreading}

In the delayed indirect spreading strategy stars survive during long periods of time, and the spreading between stars does not occur directly but using connectors as intermediaries. The proof in this case takes most of the work already done for the other mechanisms, with the sole exception being that infections on connectors have a very limited lifespan, which forces us to be a little more careful.
\medskip

In particular, we introduce a more restrictive notion of stability for connectors. For every $k\ge 1$, we introduce 
$$ \dbar \Co_k := \{y \in \bar \Co_k \colon \mathcal U^y\cap[k-1,k] \ne \emptyset\},$$
and we use these connectors in the spreading mechanism. Similarly as for stable connectors, the events 
\smash{$\{|\dbar \Co_k|> \theta' N/4\}$} hold exponentially long, where 
\smash{$\theta'= \theta(1- e^{-\kappa_0})$.} We now \emph{work conditionally on a realization of $\mathcal{U}^y,\mathcal{R}^y, y\in\Co$ such that both
\smash{$\{|\bar{\Co}_k|>\theta N/4\}$} and 
\smash{$\{|\dbar \Co_k|> \theta' N/4\}$} hold exponentially long}.
Recall the concepts of $T$-stability and $[L,T]$-susceptibility as given in Definitions~\ref{Tstable} and~\ref{suscept}. We replace the concept of $T$-infectiousness used for spreading among stars, by the following definition.
\begin{definition}
\label{TCinfectious}
We say that $x\in\St$ is $T$-$\Co$-infectious if
$$\big|\big\{k\in\N\cap[0,T-1] \colon\mathcal{R}^x\cap[k,k+1]=\emptyset\big\}\big|\;>\;\sfrac{T}{3}.$$
\end{definition}

Using Lemma~\ref{ldev} we know that each $x\in\St$ is $T$-$\Co$-infectious with probability $1-e^{-cT}$ for some   $c>0$ and hence $\liminf_{T\to \infty}\liminf_{N\rightarrow\infty}\P(x\mbox{ is }T\mbox{-}\Co\mbox{-infectious})=1$, which is why we can obtain the following result, in analogy to Lemma~\ref{starprop}.

\begin{lemma}
\label{starprop2}
There exists $q_1>0$ as in Lemma~\ref{LTsusceptibility} such that for all large $T$ and $L\in[0,T)$,
$$\liminf_{N\rightarrow\infty} \inf_{x\in\St}
\P(x\mbox{ is }T\mbox{-}\Co\mbox{-infectious, }T\mbox{-stable and }[L,T]\mbox{-susceptible})>q_1,$$
and these events are independent for different $x\in\St$.
\end{lemma}

We set $\St'_0\subset \St_0$ to be the subset of the initially infected stars that are also $T$-stable, $T$-$\Co$-infectious, $[0,T]$-susceptible and $[0,T]$-infected. If $a$ is small and $N$ large, its cardinality $S'_0$ dominates a binomial random variable with parameters $S_0$ and $q_1 q_2$. Moreover, for $x\in \St_0$, the event $x\in \St'_0$ is increasing in the processes $\mathcal I^{x,y}, y\ne x$, and thus by Harris' inequality, it is positively correlated with every event which is increasing in the processes $\mathcal I^{x,y}, y\ne x$. Further, let 
$$K:= \left\{k\in \N\cap[0,T] \colon \big|\{x \in \St'_0, \mathcal{R}^x\cap[k,k+1]=\emptyset\}\big| \ge S'_0/6\right\}.$$
We necessarily have
$$\sfrac {S'_0 T} 3 \le \sum_{k\in \N\cap[0,T]}\sum_{x\in \St'_0} 1_{\{\mathcal R^x \cap [k,k+1]=\emptyset\}} \le |K| S'_0 + (T-|K|) \sfrac {S'_0} 6,$$
where the left inequality follows from the definition of 
$T$-$\Co$-infectiousness, and the right inequality from the definition of $K$. It follows $|K|\ge T/5$. We use this set to search for times in which stars can infect sufficiently many stable connectors. More precisely, we let
$$ \mathscr P_0:= \{(k,y), k\in K, y \in \dbar \Co_k, \exists x \in \St'_0,\;  \mathcal I^{x,y}\cap[k,k+1]\ne \emptyset\}.$$

Conditionally on $k\in K$, $y\in \dbar{\Co_k}$ and $x \in \St'_0$, we have $\mathcal I_0^{x,y}\cap[k,k+1]\ne \emptyset$ with probability at least \smash{$1-e^{-\lambda}$}, thanks to the positive correlation with the event $x\in \St'_0$, and at time 
\smash{$t=\min (I_0^{x,y}\cap[k,k+1]) $} we have $\{x,y\} \in \mathscr G^{\ssup N}_t$ with probability at least $p_{x,y}\ge p(a,1)/N$, thanks to the update of $y$ on time interval $[k-1,k]$. 


Proceeding as before, we obtain that for large $T$ and large $N$ and conditionally on $S'_0$, on $K$ and on $(C_k)_{k\in K}$, the cardinality $P_0$ of $\mathscr P_0$ dominates a binomial random variable with parameters $$\left \lceil \sfrac {\theta' T N}{20}\right \rceil \mbox{  and } \sfrac {\lambda p(a,1)}{24N} S'_0 \wedge \sfrac 12.$$
Finally, we define
$\St''_0:= \{x\in \St \backslash \St_0 \colon \exists (k,y)\in \mathscr P_0,\; \mathcal I^{x,y}\cap [k+1, k+2]\ne \emptyset, X_T(x)=1\}$
and $S''_0= |\St''_0|$, and observe that $S_1 \ge S'_0 + S''_0$.
Conditionally on $\St_0$, $\St'_0$ and $\mathscr P_0$, we have, independently for each $(k,y)\in \mathscr P_0$,
$$\P(\mathcal I^{x,y}\cap [k+1, k+2]\ne \emptyset)\ge \sfrac{\lambda p(a,1)}{2N},$$
whence the probability that there exists $(k,y)\in \mathscr P_0$ such that $ \mathcal I^{x,y}\cap [k+1, k+2]\ne \emptyset$ is at least {$\sfrac {\lambda P_0 p(a,1)} {4N} \wedge \sfrac 12$.} Now, the probability that $x$ gets infected on $[0,T]$ is at least {$\sfrac {\lambda P_0 p(a,1)}{4N} \wedge \sfrac 12$}. Conditionally on this, writing $T_x$ the first time when it gets infected, we have that with probability at least $q_1 q_2$ (when $T$ and $N$ large), the star $x$ is $T$-stable, $[T_x,T]$-susceptible and $[T_x,T]$-infected, whence $X_T(x)=1$.
In other words, 
we can bound $S''_0$ from beow by a binomial random variable with parameters 
$$\lfloor aN\rfloor -S_0 \mbox{  and } q_1 q_2\big(\sfrac {\lambda p(a,1)}{4N} P_0 \wedge \sfrac 12\big).$$
Gathering the results, when $T$ is large we can use Lemma~\ref{discreteMC2} with $n=\lfloor aN \rfloor$, $\rho=q_1 q_2$, $\rho'=\sfrac {\lambda}{24}  a p(a,1)$, $\rho''=\lambda q_1 q_2 a p(a,1) /4$, and $\rho'''=\theta' T /20 a$, under the condition $\rho \rho' \rho'' \rho'''>1$. This condition is satisfied if $\lambda^2 T(a,\lambda) a p(a,1)^2$ is large enough, which concludes the proof of the delayed indirect spreading case of Theorem~\ref{teolower}.

\medskip

\section{Fast extinction and upper bounds}

\medskip

To obtain upper bounds we need to show that no mechanism can outperform the ones examined in the lower bounds. This cannot be done explicitly, but requires a more abstract supermartingale argument, which we now introduce. We start by coupling our process to a simpler process which is a stochastic upper bound.  
%
%

\subsection{A  coupling}
We construct a coupling between the contact process on the dynamic network, described by the pair of processes~$(X,\mathscr G^{\ssup N})$, and a process~$Y$, 
which we call the `wait-and-see' process. 
The process $(Y_t\colon t>0)$ takes values in $\{0,1\}^N \times \{0,1\}^{N\otimes N}$, where $N \otimes N$ is the set of potential edges, ie.\ unordered pairs of distinct vertices in $\{1,\ldots,N\}$.  We say a vertex $x$ is infected at time $t$ (for the wait-and-see process) if $Y_t(x)=1$, and we say a potential edge $\{x,y\}$ is \emph{revealed} at time~$t$ if $Y_t(x,y)=1$. Informally, a potential edge is unrevealed at time $t$ if we have no information about its presence in the dynamic network $\mathscr G^{\ssup N}_t$.
%
The wait-and-see model evolves according to the following rules:
\begin{itemize}
\item Every infected vertex $x$ recovers at rate 1.
\item If $x$ is infected then it infects every uninfected vertex $y$, 
\begin{itemize}
\item with rate $\lambda$ if $\{x,y\}$ is revealed (ie.\ if $Y_t(x,y)=1$) and 
\item with rate $\lambda p_{x,y}$ if it is unrevealed (ie.\ if $Y_t(x,y)=0$). 
\end{itemize}
In the latter case, when $x$ infects $y$, the value of $Y_t(x,y)$ immediately turns to 1. 
\item If $x$ and $y$ are both infected and $\{x,y\}$ is unrevealed, it gets revealed at rate~$\lambda p_{x,y}$.
\item Finally, each vertex updates at rate $\kappa_x$. Updating of $x$ means that all 
its adjacent potential edges turn to unrevealed.
\end{itemize}

\begin{lemma}
Fix deterministic initial conditions $X_0(v)\le Y_0(v)$ for all $v\in\{1,\ldots, N\}$.
There exists a coupling of
\begin{itemize}
\item the dynamic random network $(\mathscr G^{\ssup N}_t \colon t\geq 0)$, 
\item the original infection process on this network $(X_t \colon t\ge0)$, and 
\item the wait-and-see process $(Y_t \colon t\ge 0)$, started from vertices in $Y_0$ infected and all its edges unrevealed, 
\end{itemize}
such that, at all times~$t\ge0$, we have  $X_t(v)\le Y_t(v)$ for all $v\in\{1,\ldots, N\}$, 
and every revealed edge is an edge in \smash{$\mathscr G^{\ssup N}_t$.}
\end{lemma}

The proof of this lemma is similar to the proof of Proposition~6.1 in \cite{JM15},
so we omit it.




\subsection{Proof of Theorem \ref{teoupper}.}

In this section we prove Theorem~\ref{teoupper} for the vertex component of~$Y$, and hence for $X$ since we have that $X\leq Y$ stochastically. We use the function~$S$, given in the assumptions of the theorem, to define a function~$m_t$ which, based on the state of $Y$, attaches a score to every vertex. We take the accumulated score of the vertices in the network, raise to the power~$\delta$, and subtract a suitable drift to get a supermartingale. We then exploit optional stopping at the extinction time to get upper bounds and prove both statements of the theorem.
\medskip

To prove the first statement, assume there exists $S$ and $D$ satisfying \eqref{condD}, \eqref{defiT} and \eqref{condS1}, and introduce the notation $s(x)=S(x/N)$ and $t(x)=c[T_{\lambda}(x/N)\kappa_x]^{-1}s(x)$, for vertices 
$x\in \{1,\ldots, N\}$, where \smash{$c=\min\{\kappa_0,\frac{\kappa_0^2}{16 c_2},\frac{1}{4}\},$}
which by hypothesis is larger than~$4D$. Given these functions, we define the score of a configuration as $$M_t:=\sum_{x=1}^N m_t(x),$$ where\vspace{-3mm}
$$
m_t(x)= \left\{ \begin{array} {ll}
s(x) + (2\kappa_x^{-1} \lambda N_t(x) \wedge \frac 12) (2 t(x)) &\text{ if }Y_t(x)=1, \\
(2\kappa_x^{-1} \lambda N_t(x) \wedge 1) (s(x)+t(x))& \text{ if }Y_t(x)=0.
\end{array}\right.
$$
and \smash{$N_t(x)=\sum_{y\neq x}Y_t(x,y)$} is the number of revealed neighbours of $x$ at time~$t$. 
\smallskip

Even though it may seem a bit obscure at first glance, the score is actually natural; every infected vertex has a base score of $s(x)$ or $0$ (depending on whether it is infected or not) and $m_t(x)$ increases linearly on the amount of its revealed neighbours, which reflects the fact that revealed neighbours make the propagation of the infection easier. In both cases the score grows linearly up until some maximal cap at which $m_t$ is the same for infected and non-infected vertices; a natural choice since from a certain amount of revealed neighbours on, we can think of vertices as permanently infected.\smallskip

Our aim is to prove that $Z_t:=M_t^{\delta}+\delta\omega t$ defines a supermartingale, for some suitable~$\omega$, by showing that the expected infinitesimal change of $M_t$ is less than 
\smash{$-\omega M_t^{1-\delta}$. } We start with a few basic observations on the score function $m_t$ and on the functions $t$ and $s$.
\begin{itemize}
\item From its definition and our choice of $c$, we have $t(x)\leq s(x)$ and from the hypothesis on $S$, we have that \smash{$s(x)^{1-\delta}\leq \frac{\mathfrak{c}}{c}\kappa_xt(x)$.}
\item The score of a vertex is monotone with respect to the value of $Y_t(x)$ and of $N_t(x)$. 
\item The maximal value $m_t(x)= s(x) + t(x)$ is obtained if either $Y_t(x)=1$ and $N_t(x)\ge \lambda^{-1} \kappa_x/4$, or $Y_t(x)=0$ and $N_t(x)\ge \lambda^{-1} \kappa_x/2.$ 
\end{itemize}
Based on these observations, we can easily obtain upper bounds for the expected infinitesimal change in the score $m_t(x)$, knowing the infection process up to time $t$. We give the bounds depending on the values of $Y_t(x)$ and $N_t(x)$.\medskip

$\mathbf{ (i)\ Y_t(x)=1,\ N_t(x)\ge \lambda^{-1}\kappa_x/4.}$ In this case the score of $x$ can only decrease (or remain unchanged) with each possible change, so we obtain the bound considering only an update event at $x$, which yields
$$\frac 1 {dt} \E[m_{t+dt}(x)-m_t(x) \vert \F_t]\le -\kappa_x t(x).$$

$\mathbf{(ii)\ Y_t(x)=1,\ N_t(x)< \lambda^{-1}\kappa_x/4.}$ 
In this case we can bound the infinitesimal change by the expression
$$\Big(2\kappa_x^{-1}\lambda N_t(x)[s(x)-t(x)]-s(x)\Big)+4\lambda^{2}\kappa_x^{-1}t(x)\sum_{y \colon Y_t(x,y)=0}p_{x,y},$$
where the first term comes from the recovery at $x$ and the second one from the possible revealing of a neighbouring edge. As $t(x)\le s(x)$ and $N_t(x)< \lambda^{-1}\kappa_x/4$ the first term is bounded by 
$$-\frac{s(x)+t(x)} 2\le -\frac {s(x)} 2 \le -\frac 1 {2c} \kappa_x t(x) T_\lambda(\sfrac x N).$$ 
 On the other hand, since $\sum_{y}p_{x,y}$ can be bounded by $\int p(\sfrac x N, t) dt \le c_2\left(\sfrac{x}{N}\right)^{-\gamma}$, we can bound the second term by
$$4 c_2\kappa_0^{-2}(\kappa_x t(x))\left[\lambda^{2}\left(\frac{x}{N}\right)^{-\gamma+2\gamma\eta}\right]\;\leq\;4 c_2\kappa_0^{-2}\kappa_x t(x)T_{\lambda}\left(\sfrac{x}{N}\right),$$
so by our choice of $c$ we obtain
$$\frac 1 {dt} \E[m_{t+dt}(x)-m_t(x) \vert \F_t]\le -\frac{\kappa_x t(x) T_\lambda(\frac x N)}{4c}\le -\kappa_x t(x).$$

$\mathbf{(iii)\ Y_t(x)=0,\ N_t(x)\ge \lambda^{-1}\kappa_x/2}$. As in the first scenario, the score is maximal in this case. We obtain the bound again considering only an update event at $x$, which yields
$$\frac 1 {dt} \E[m_{t+dt}(x)-m_t(x) \vert \F_t]\le -\kappa_x m_t(x).$$

$\mathbf{(iv)\ Y_t(x)=0,\ N_t(x)\le \lambda^{-1}\kappa_x/2}$. In this case we can bound every positive increment of $m(x)$ by the maximal score $s(x)+t(x)$, hence we can bound the infinitesimal change by
$$-\kappa_x m_t(x)
+ \lambda N_t(x)[s(x)+t(x)]
+ \sum_{y \colon Y_t(y)=1}\lambda p_{x,y}[s(x)+t(x)].$$
where the first term comes from possible updates of $x$, the second from infections coming through neighbouring revealed edges and the third from infections coming through unrevealed edges. Since $\ N_t(x)\le \lambda^{-1}\kappa_x/2$, the second term is exactly $\frac{\kappa_x}{2} m_t(x)$ and hence, since $t(x)\leq s(x)$, we obtain
$$\frac 1 {dt} \E[m_{t+dt}(x)-m_t(x) \vert \F_t]\le-\frac{\kappa_x}{2} m_t(x)
+ 2\lambda\sum_{y\colon Y_t(y)=1} p_{x,y}s(x).$$

Now, we can consider the whole score, and write
%
\begin{align*}
\frac 1 {dt} & \E[M(t+dt) -M(t)\vert \F_t]= \sum_x  \frac 1 {dt} \E[m_{t+dt}(x)-m_t(x) \vert \F_t] \\
& \le \sum_{\heap{x\colon Y_t(x)=0}{N_t(x)\ge \lambda^{-1} \kappa /2} } -
\kappa_x m_t(x) 
+ \sum_{x \colon  Y_t(x)=1} -\kappa_x t(x) \\
& + \sum_{\heap{x\colon Y_t(x)=0}{N_t(x)< \lambda^{-1} \kappa /2}} -\frac{\kappa_x}{2} m_t(x) 
 + \sum_{\heap{x\colon Y_t(x)=0}{N_t(x)< \lambda^{-1} \kappa /2}} 2\lambda \sum_{y\colon Y_t(y)=1}  p_{x,y}s(x).
\end{align*}
For the last term, we can reverse the role of $x$ and $y$ obtaining the expression 
$$\sum_{x \colon Y_t(x)=1} 2\lambda \sum_{y} p_{x,y}s(y)\;\leq\;2 D\sum_{x \colon Y_t(x)=1}\frac{s(x)}{T_{\lambda}(x/N)}\;\leq\;\frac 12 \sum_{x \colon Y_t(x)=1} \kappa_x t(x),$$
where the first inequality comes from our hypothesis on $S$ for sufficiently large $N$.

We thus arrive at
\begin{equation}\label{ineqmart}\frac 1 {dt} \E[M(t+dt)-M(t)\vert \F_t]\le 
-\;\frac{1}{2}\sum_{x \colon Y_t(x)=0}\kappa_x m_t(x)\;- \frac{1}{2}\sum_{x \colon Y_t(x)=1}  \kappa_x t(x),
\end{equation}
which is clearly negative and hence $(M(t) \colon t\geq 0)$ is a supermartingale. To show that $(Z_t \colon t\ge0)$ is a supermartingale note that the second term satisfies
\begin{equation}\label{step1}-\;\frac 12 \sum_{x\colon Y_t(x)=1}\kappa_xt(x)\leq  -\; \frac 12 \sum_{x \colon Y_t(x)=1}\frac{c}{\mathfrak{c}}s(x)^{1-\delta}\leq -\;\frac{c}{4\mathfrak{c}}\sum_{x \colon Y_t(x)=1}m_t(x)^{1-\delta}.
\end{equation} 
On the other hand, since $S$ is bounded from below by some positive constant, say $s_0$, we can bound the first term as follows,
$$-\;\frac{1}{2}\sum_{x\colon  Y_t(x)=0}\kappa_x m_t(x)\;\leq\;-\;\frac{1}{2}\sum_{\substack{x\colon Y_t(x)=0\\ m_t(x)\neq 0}}2\lambda s(x)\;\leq\;-\;\lambda s_0^{\delta}\sum_{\substack{x\colon Y_t(x)=0\\ m_t(x)\neq 0}}s(x)^{1-\delta}$$
where the last inequality comes from $s(x)/s_0>1$. Since $m_t(x)\leq 2s(x)$ we deduce that this term is bounded by \smash{$-\frac12 \lambda s_0^{\delta}\sum_{x: Y_t(x)=0}m_t(x)^{1-\delta}$} so, together with \eqref{step1} we conclude that there exists some $\omega>0$ depending on $\lambda$ but not on $N$ such that
\begin{equation}\label{desM}\frac 1 {dt} \E[M(t+dt)-M(t)\vert \F_t] \le - \omega\sum_x m_t(x)^{1-\delta}\le -\omega M(t)^{1-\delta},\end{equation}
where the last inequality is due to $0<\delta<1$, and hence \smash{$Z_{t\wedge T_{\rm ext}}$} defines a positive supermartingale which converges almost surely to $Z_{T_{\rm ext}}$. Since $T_{\rm ext}$ is increasing in the initial condition of $Y$ it is enough to take $Y(0)=1$ to prove the theorem. We infer from the optional stopping theorem that
\begin{equation}
\label{desfinalproof1}
\delta \omega\, \E[T_{\rm ext}]\;=\;\E[Z_{T_{\rm ext}}]\;\leq\;\E[Z_0]\;=\;\E\big[M_0^{\delta}\big]\;=\;N^{\delta}\Big[\frac{1}{N}\sum_{x=1}^{N}s(x)\Big]^{\delta},
\end{equation}
but the expression inside the brackets converges to $\int_0^1 S(x)dx$ which is a fixed constant and hence the first statement of Theorem~\ref{teoupper} is proved.\medskip

In order to prove the second statement we use the duality described in Proposition~\ref{dualprop} 
to deduce \smash{$I_N(t)=\frac{1}{N}\sum_{x=1}^N\P_x(X_t\not=0)$} where $\P_x$ corresponds to the law of the process with initial condition $X_0=\delta_x$. Since $Y_t$ stochastically bounds $X_t$ from above, $I_N(t)\leq$ \smash{$\frac{1}{N}\sum_{x=1}^N\P_x(t<T_{\rm ext})$}, where in this context $T_{\rm ext}$ is the extinction time of $Y$.
Defining $T_{\rm hit}$ as the first time that $Y_t(x)=1$ for some $x\leq \lceil aN\rceil$, and $T:=T_{\rm ext}\wedge T_{\rm hit}$, we obtain 
\begin{equation}\label{des1fin2} \P(t<T_{\rm ext})\le \P(T_{\rm hit}<T_{\rm ext})+\P(t<T),
\end{equation}
which leads us to use the stopped process $Y_{t\wedge T}$ instead of $Y_t$. 

We extend the function $S$ defined on $[a,1]$ to the whole interval $[0,1]$ by setting 
$$\bar S(x) = S(x\vee \sfrac {\lceil a N\rceil} N).$$ This function satisfies condition \eqref{defiT} for $x\ge a$ as well as condition \eqref{condS1} for $x\in\{\lceil a N \rceil +1, \ldots, N\}$. 
A close look at the latter proof shows that if we start with $Y_0=0$ on $\{1,\ldots, \lceil aN\rceil\}$ and define $s$, $t$, $m$, and $M$ and $Z$ as before, then the stopped processes $M_{t\wedge T}$ and $Z_{t\wedge T}$ are positive supermartingales.
As $M_{t\wedge T}$ converges almost surely to $M_{T}=M_{T_{\rm hit}}{\bf1}_{\{T_{\rm hit}<T_{\rm ext}\}}\geq s(\lceil aN\rceil){\bf1}_{\{T_{\rm hit}<T_{\rm ext}\}}$, by the optional stopping theorem we get 
\begin{equation}\label{des2fin2}
\P_x(T_{\rm hit}<T_{\rm ext})\le \frac {\E_x(M_0)} {s(aN)}=\frac {s(x)} {s(\lceil aN\rceil)}.
\end{equation}
On the other hand, since $Z_{t\wedge T}$ is a positive supermartingale, it converges almost surely to \smash{$Z_T=M_T^{\delta}+\delta\omega T\geq \delta\omega T,$} so by the optional stopping theorem,
\begin{equation}\label{des3fin2}\P_x(t<T)\;\leq\;\sfrac{1}{t}\E_xT\;\leq\;\sfrac{1}{\delta\omega t}\E_x[Z_0]\;=\;\sfrac{s(x)^{\delta}}{\delta\omega t}.
\end{equation}
Gathering \eqref{des1fin2}, \eqref{des2fin2} and \eqref{des3fin2}, and bounding the probability of survival by 1 whenever $x\leq \lceil aN\rceil$, we obtain
$$I_N(t)\;\leq\;\frac{\lceil aN\rceil}{N}\;+\;\frac{1}{Ns(\lceil aN\rceil)}\sum_{x=\lceil aN\rceil +1}^{N}s(x)\;+\;\frac{1}{\delta\omega t N}\sum_{x=\lceil aN\rceil +1}^{N}s(x)^{\delta}.$$
Noticing that all these terms converge when $N\rightarrow\infty$, we have that for all $N$ and $t\ge0$,
$$I_N(t)\;\leq\;a\;+\;\frac{1}{S(a)}\int_a^1S(x)dx\;+\;\frac{1}{\delta\omega t}\int_a^1S(x)^{\delta}dx+ \eps(N),$$
with $\eps(N)\to 0$, which concludes the result when taking $t$ large enough (note both integrals are finite since $S$ is continuous on $[a,1]$). 

\section{Application to the factor kernel}

In this section we prove the first half of Theorem~\ref{teofinal} applying both Theorem~\ref{teolower} and \ref{teoupper} to the factor kernel $p(x,y)=\beta x^{-\gamma}y^{-\gamma}$. Our proof is structured as follows:
\begin{enumerate}
\item For each of the four strategies in Theorem~\ref{teolower} we find the function $a(\lambda)$ of maximal order satisfying the respective condition given in the theorem.\\[-3mm]
\item We define $a_0(\lambda)$ as the maximum of these functions over the four strategies. This function gives for each  $\lambda$ the definition of the set of stars. Using Theorem \ref{teolower} we derive from it a lower bound of order $\lambda a_0 p(a_0,1) \wedge 1$, or more simply $\lambda a_0(\lambda)^{1-\gamma}$, for the lower metastable density. 
\\[-3mm]
\item We search for a nonincreasing function $S$ and for $a_1=a_1(\lambda) \in[0,1]$ as small as possible such that the following inequality is satisfied for small $\lambda$,
\begin{equation}
\label{int_condS2}
 \lambda T_\lambda(x) \int_0^1 p(x,y) S(y \vee a_1) dy \le D S(x),  \ x\in[a_1,1].
\end{equation}
Observe that this is a continuous version of Inequality~\eqref{condS2}, and actually implies Inequality~\eqref{condS2}, as by the monotonicity of $p$ and $S$, we have, for $i\ge a_1 N$,
$$\lambda T_{\lambda}\left(\sfrac{i}{N}\right)\sum_{j=1}^Np_{i,j}S\left(\sfrac{j \vee \lceil a_1 N\rceil}{N}\right)\leq \lambda T_{\lambda}\left(\sfrac{i}{N}\right)
\left(\sfrac{i}{N}\right)^{-\gamma}\int_0^1S(y\vee a_1)y^{-\gamma}\, dy.$$
If we can take $a_1=0$ in~\eqref{condS2}, then we apply Theorem~\ref{teoupper}, (1), and deduce fast extinction. In the other cases we will always have proven already metastability, and Theorem~\ref{teoupper}, (2) then gives us an upper bound on the upper metastable density~as
$$ \rho^+(\lambda) \le a_1(\lambda) +\frac 1 {S(a_1(\lambda))} \int_{a_1(\lambda)}S(y) dy.$$
Note we have not discussed how to choose the function $S$. This will be further discussed in the examples below.
\end{enumerate}
To avoid cluttered notation we henceforth assume $\beta=1$, which does not affect the results.%

\subsubsection*{The function $a_0(\lambda)$}
\label{fk}

Our aim is to find for each strategy the maximal function satisfying the respective condition in 
Theorem~\ref{teolower}. 
\begin{itemize}
\item {\bf Quick Direct Spreading}: We study the expression
$\lambda ap(a,a)=\lambda a^{1-2\gamma}$ and check whether it is bounded away from zero. If $\gamma\leq 1/2$ this is never satisfied. If $\gamma>1/2$ we impose that the expression be constant 
and obtain $$a(\lambda)=r\lambda^{\frac{1}{2\gamma-1}} \quad \mbox{ for some $r>0$.}$$
\item {\bf Delayed Direct Spreading}: We study the expression
$\lambda Tap(a,a)=\lambda T a^{1-2\gamma}$ with
$T$ as in Theorem~\ref{teolower}.
To facilitate our study we impose $T\rightarrow\infty$ instead of just being large, but this translates into \smash{$T\log^2(T)=C\lambda^2 a^{-\gamma(1-2\eta)}\rightarrow\infty$} which can only occur if $\eta<1/2$. 
To ensure boundedness of the expression from zero we observe that 
$\lambda T a^{1-2\gamma}\;\leq\;\lambda^3a^{1-3\gamma+2\gamma\eta}$
and hence it suffices that $1-3\gamma+2\gamma\eta<0$. Assuming this and $\eta<1/2$ we find $a(\lambda)$ by imposing that $\lambda T a^{1-2\gamma}$ be constant, say equal to $c$,
and deduce from the definition of $T$ that
$$\left[\log(c)-\log\left(\lambda a^{1-2\gamma}\right)\right]^2\;=\;\sfrac{C}{c}\lambda^3a^{1-3\gamma+2\gamma\eta}.$$
Since $T\rightarrow\infty$, the expression on the left goes to infinity, and hence $a$ is of the form $$a(\lambda)=\lambda^{\frac{3}{3\gamma-2\gamma\eta-1}}f(\lambda),$$ 
for some function $f$ going to zero as $\lambda\rightarrow 0$. Replacing this new expression for $a$ we obtain
$[\log(c)-c_1\log(\lambda)-c_2\log(f(\lambda))]^2=\frac{C}{c}f(\lambda)^{1-3\gamma+2\gamma\eta}.$
The expression on the right tends to infinity polynomially in $f$, so the equality holds only if $\log\lambda$ dominates $\log f(\lambda)$ giving $f(\lambda)$ of order \smash{$[-\log(\lambda)]^{\frac{2}{1-3\gamma+2\gamma\eta}}$}. 
We finally get the maximal $a$ of the form
$$a(\lambda)\;=\;r\Big[\frac{\lambda^3}{(\log\lambda)^2}\Big]^{\frac{1}{3\gamma-2\gamma\eta-1}}
\quad \mbox{ for some $r>0$.}$$
\item {\bf Quick Indirect Spreading}: We study the expression
\smash{$\lambda^2 a^{1-\gamma}p(a,1)$}
\smash{$=\lambda^2 a^{1-2\gamma}$}. If $\gamma\leq 1/2$ this inevitably tends to zero and 
otherwise we obtain that the maximal $a$ is of the form
$$a(\lambda)=r\lambda^{\frac{2}{2\gamma-1}}\quad \mbox{ for some $r>0$.}$$
\item {\bf Delayed Indirect Spreading}: In this case, we need to consider the expression
\smash{$\lambda^2 Ta^{1-\gamma}p(a,1)=\lambda^2 T a^{1-2\gamma}$}
whose study is analogous to what was done in the delayed direct spreading case, obtaining that condition~(iv) can only hold when $\eta<1/2$ and $1-3\gamma+2\gamma\eta<0$ and in this case $a$ must be of the form
$$a(\lambda)\;=\;r\left[\frac{\lambda^4}{(-\log(\lambda))^2}\right]^{\frac{1}{3\gamma-2\gamma\eta-1}}\quad \mbox{ for some $r>0$.}$$
\end{itemize}
For the final form of $a_0$, we notice first that on the set
\begin{equation}
\label{regfast}
\Big\{(\gamma,\eta) \colon \gamma\leq \sfrac{1}{3-2\eta},\;\eta\leq \sfrac{1}{2}\;\mbox{ or }\;\gamma\leq \sfrac{1}{2},\;\eta\geq \sfrac{1}{2}\Big\}
\end{equation}
none of the conditions of Theorem~\ref{teolower} hold, so we expect that fast extinction occurs for parameters inside this region. For the construction of $a_0$ on the complement of this set, we note that all of the functions $a$ have the form $\lambda^{\mathfrak{e}'+o(1)}$, and since $\lambda<1$, for each $\lambda$ small the dominant survival strategy (that is, the one which gives the largest lower bound for the density) corresponds to the expression with the smallest exponent. If $\eta<1/2$ this gives 
$$a_0(\lambda)\;=\;\left\{\begin{array}{cl}r\left[\frac{\lambda^3}{(-\log(\lambda))^2}\right]^{\frac{1}{3\gamma-2\gamma\eta-1}}&\mbox{ if }\frac{1}{3-2\eta}<\gamma<\frac{2}{3+2\eta},\\&\\
r\lambda^{\frac{1}{2\gamma-1}}&\mbox{ if }\frac{2}{3+2\eta}<\gamma<1,
\end{array}\right.$$
while in the case $\eta>1/2$ we obtain
$$a_0(\lambda)=r\lambda^{\frac{1}{2\gamma-1}}\;\mbox{ if }\;\gamma>\sfrac{1}{2}.$$
Computing the lower bound for the density as $\lambda a_0^{1-\gamma}$ for the values of the parameter where $a_0$ is defined gives the lower bound in~\eqref{dens1}.

\subsubsection*{The function $S$}

In the case of the factor kernel, Inequality~\eqref{int_condS2} takes the simple form
$$ \lambda T_\lambda(x) x^{-\gamma} \int_0^1 y^{-\gamma} S(y \vee a) dy \le D S(x).
$$
Since the integral does not depend on $x$, a natural choice is to consider the function
$$ S(x)=T_\lambda(x) x^{-\gamma}.$$
This scoring function is also somehow natural, as we now explain. The average degree of $x$ is of order $x^{-\gamma}$ and if $x$ is infected, we should wait on average at most time $T_\lambda(x)$ before $x$ turns to healthy and surrounded by unrevealed edges. More precisely, this should happen roughly at time $\max\{T(x,\lambda),1\}$, where $T(x,\lambda)$ is as in~\eqref{T(a,lambda)}, but this is bounded by $T_\lambda(x)$, and even of the same order, up to logarithmic terms\footnote{The discrepancy between the time-scale function $T_\lambda(x)$ used in Theorem~\ref{teoupper} and that in Theorem~\ref{teolower} explains why, as we will see, the lower bounds for $\rho^-(\lambda)$ and the upper bounds for $\rho^+(\lambda)$ that we get match only up to logarithmic terms. They match indeed up to a constant multiplicative term when $T_\lambda(x)=1$.}. Thus $\lambda S(x)$ should be a reasonable upper bound for the average number of infections sent by vertex $x$ before the first time when it is healthy and surrounded by unrevealed edges (namely the first time $t$ for which $m_t(x)=0$), and thus $S$ seems a reasonable scoring function.
\smallskip

Using this choice of $S$, Inequality~\eqref{int_condS2} becomes
$$ \lambda \int_0^1 y^{-\gamma} S(y \vee a) dy \le D.$$
Using that $T_\lambda(x)\le 1+\lambda^2 x^{-\gamma(1-2\eta)}$, the left hand side can be bounded as follows,
\begin{eqnarray*}
\lambda \int_0^1 y^{-\gamma} S(y \vee a) dy &=& \lambda a^{1-2\gamma} T_\lambda(a)+ \lambda \int_a^1 S(y) y^{-\gamma} dy \\
&\le& \lambda a^{1-2\gamma} T_\lambda(a)+ \lambda^3 \int_a^1 y^{-\gamma(3-2\eta)} dy + \lambda \int_a^1 y^{-2\gamma} dy \\
&\le& \rho (\lambda^3 a^{1-\gamma(3-2\eta)}+\lambda a^{1-2\gamma}+ \lambda),
\end{eqnarray*}
for some constant $\rho>0$ depending only on $\gamma$, $\eta$, under the hypothesis $\gamma\notin\{\sfrac 12, \sfrac 1 {3-2\eta}\}$. Note that in the last inequality we have used that 
\smash{$\int_a^1 y^\iota dy \le \sfrac 1 {|1+\iota|}(1+a^{1+\iota})$} for $\iota\ne -1$, and the upper bound is sharp up to a multiplicative constant. \smallskip

We now consider three different cases.\smallskip 

\noindent
{\bf (1)\ The case $\eta<\frac{1}{2}$ and $\gamma<\frac{1}{3-2\eta}$, or $\eta\ge \frac 12$ and $\gamma<\frac 12$.}

\noindent
Here we can take $a_1=0$. Indeed, we have
$\lambda \int_0^1 y^{-\gamma} S(y) dy \le \rho \lambda$, therefore Inequality~\eqref{int_condS2} is satisfied if $\lambda\le D/\rho$.
Moreover, the function $S$ satisfies $T_{\lambda}(x)\leq S(x)^{\delta}$ for
\begin{itemize}
\item $\delta\geq\frac{1-2\eta}{2-2\eta}$ whenever $0\leq\eta<1/2$,
\item $\delta>0$ in the case $\eta>1/2$.
\end{itemize}
Using Theorem~\ref{teoupper}, (1), we deduce fast extinction for small $\lambda$. More precisely, when $\eta<1/2$, we get \smash{$\E[T_{\rm ext}]\leq \omega'N^{_{\frac{1-2\eta}{2-2\eta}}}$} for some $\omega'<+\infty$. In the case $\eta\ge 1/2$, for every $\delta>0$, there is some $\omega'<+\infty$ such that \smash{$\E[T_{\rm ext}]\leq \omega'N^\delta$}. In particular, the extinction time grows even slower than polynomially.
\smallskip

\noindent
{\bf (2)\ The case $\eta<\frac{1}{2}$ and $\frac{1}{3-2\eta}<\gamma<\frac 2 {3+2\eta}$.}
\smallskip

\noindent
A sufficient condition for Inequality~\eqref{int_condS2} to be satisfied is
$$ \max(\lambda^3 a^{1-\gamma(3-2\eta)},\lambda a^{1-2\gamma}, \lambda)\le D/3\rho.$$
This requires, in particular, that 
$$a\ge a_1(\lambda):= r \lambda^{\sfrac 3 {3\gamma-2\eta-1}},$$
where $r=(D/(3\rho))^{\sfrac 1 {3\gamma-2\eta-1}}<\infty$.
One can check this is also a sufficient condition for small $\lambda$ when $\gamma<\frac 2 {3+2\eta}$.
Applying Theorem~\ref{teoupper}, (2), we deduce that
$$\rho^+(\lambda)\;\leq\; a_1(\lambda)+\frac{1}{S(a_1(\lambda))}\int_{a_1(\lambda)}^1S(y)\, dy.$$
From now on, we write $f(\lambda)\lesssim g(\lambda)$ if the function $g(\lambda)/f(\lambda)$ is bounded from below by a positive constant, and similarly for $\gtrsim$.  One can check the following:
\begin{eqnarray*}
\int_{a_1(\lambda)}^1S(y)\, dy &\lesssim& \int_{a_1(\lambda)}^1 \lambda^2 y^{-\gamma(2-2\eta)}+y^{-\gamma}\, dy \\
&\lesssim& \lambda^2 a_1(\lambda)^{1-2\gamma+2\gamma \eta} +1 
\lesssim \lambda^{\sfrac {1+2\gamma \eta} {3\gamma - 2\gamma \eta -1}}+1 \lesssim 1.
\end{eqnarray*}
This, together with $S(a_1(\lambda))\ge \lambda^2 a_1^{-\gamma(2-2\eta)} \gtrsim \lambda^{- \sfrac {2-2\gamma \eta} {3\gamma-2\gamma \eta-1}}$, gives
$$ \rho^+(\lambda) \le c \lambda^{ \sfrac {2-2\gamma \eta} {3\gamma-2\gamma \eta-1}},$$
for some constant $c<\infty$. Thus we obtain the upper bound \eqref{dens1} in this region.
\smallskip

\noindent
{\bf (3)\ The case $\eta\ge \frac{1}{2}$ and $\gamma>1/2$, or $\eta\ge 1/2$ and $\gamma>1/2$.}
\smallskip

\noindent
Similarly as in the previous case, it suffices for small $\lambda$ to require
$$a\ge a_1(\lambda):= r \lambda^{\sfrac 1 {2\gamma-1}},$$
with $r=(D/(3\rho))^{\sfrac 1 {2\gamma-1}}.$
We have $S(x)=x^{-\gamma}$ for $x\ge a_1(\lambda)$, and Theorem~\ref{teoupper}, (2), yields
$$ \rho^+(\lambda)\le a_1^{\gamma} \int_0^1 y^{-\gamma} dy\le c \lambda^{\sfrac \gamma{2\gamma-1}},$$
for some constant $c<\infty$. This gives again the upper bound~\eqref{dens1} and concludes our study of the metastable densities for the factor kernel.

\addtocontents{toc}{\setcounter{tocdepth}{1}}
\section{Application to the preferential attachment kernel}

In this section we derive results for the preferential attachment kernel given by
$p(x,y)=\beta \min\{x,y\}^{-\gamma}\max\{x,y\}^{\gamma-1}$ following the programme set out at the beginning
of Section~6. Again, we assume $\beta=1$ for simplicity. Taking into account that
$p(a,a)=a^{-1}$ and $p(a,1)=a^{-\gamma}$ straightforward calculations allow us to deduce the values of the maximal order functions $a(\lambda)$ summarised in the table below. The third column gives the conditions needed to define the maximal function~$a(\lambda)$.
\smallskip

\begin{center}
\begin{tabular}{|c|c|c|}
\hline
\hline
\phantom{$^{H^H}_{y_y}$}Strategy\phantom{$^{H^H}_{y_y}$}&$a(\lambda)$&Condition\\
\hline
\hline
\phantom{$^{H^H}_{y_y}$}{\bf Quick Direct Spreading}\phantom{$^{H^H}_{y_y}$}&$--$&$--$\\
\hline
{\bf Delayed Direct Spreading}&$\left(\frac{\lambda^3}{(-\log(\lambda)^2)}\right)^{\frac{1}{\gamma(1-2\eta)}}$&$\eta<1/2$\\
\hline
{\bf Quick Indirect Spreading}&$\lambda^{\frac{2}{2\gamma-1}}$&$\gamma>1/2$\\
\hline
{\bf Delayed Indirect Spreading}&$\left(\frac{\lambda^4}{(-\log(\lambda)^2)}\right)^{\frac{1}{3\gamma-2\gamma\eta-1}}$&$\eta<\frac{1}{2}$ and $\gamma>\frac{1}{3-2\eta}$\\
\hline
\end{tabular}
\end{center}
\medskip

\noindent
Taking the maximum over permissible strategies we deduce $a_0(\lambda)$.
If $\eta<1/2$ we~get
$$a_0(\lambda)\;=\;\left\{\begin{array}{cl}r\left[\frac{\lambda^3}{(-\log(\lambda))^2}\right]^{\frac{1}{\gamma(1-2\eta)}}&\mbox{ if }0<\gamma<\frac{3}{5+2\eta},\\&\\r\left[\frac{\lambda^4}{(-\log(\lambda))^2}\right]^{\frac{1}{3\gamma-2\gamma\eta-1}}&\mbox{ if }\frac{3}{5+2\eta}<\gamma<\frac{1}{1+2\eta},\\&\\
r\lambda^{\frac{2}{2\gamma-1}}&\mbox{ if }\frac{1}{1+2\eta}<\gamma<1.
\end{array}\right.$$
If $\eta>1/2$  and $\gamma>1/2$ we get $a_0(\lambda)=r\lambda^{\frac{2}{2\gamma-1}}$, but if
$\eta>1/2$ and $\gamma<1/2$ none of the  strategies succeed. Calculating $\lambda a_0(\lambda)^{1-\gamma}$ gives the lower bounds as in~\eqref{dens2} and there is slow extinction for all parameters except for the case $\eta>1/2$ and $\gamma<1/2$, as expected. 
\medskip

To get upper bounds in the case of the preferential attachment kernel, 
the choice of a scoring function $S$ is much more delicate.
Our initial approach has been to search for a function $S$ giving equality in~\eqref{int_condS2}, or in the related Fredholm equation of the second kind
\begin{equation}\label{freddie}
\int_{a_1(\lambda)}^1 T_{\lambda}(x)p(x,y)S(y)\, dy\;=\;\sfrac{D}{\lambda}S(x),
\end{equation}
as such an~$S$ is a plausible candidate to give the best possible bounds in~Theorem~\ref{teoupper}. 
To carry out this programme requires extensive calculations with Bessel functions and modified Bessel functions. However, it turns out that relatively crude approximations to these functions also suffice 
and this is the approach we now follow.
\smallskip

\subsubsection*{The upper bound for $\gamma>1/2$.}
We start with the case $\gamma>1/2$, and we define $S$ by 
$$S(x)= T_\lambda(x) (x^{\gamma-1}+ \rho \lambda x^{-\gamma}),$$ 
where $\rho>0$ is a constant to be chosen later. To argue why $S$ may be a ``reasonable scoring function'', it is useful to note that the cardinality of the sets
$\{y\le x \colon y\sim x\}$ and $\{y\le x \colon\exists z\ge x, y\sim z\sim x \}$ are of order $x^{\gamma-1}$ and $x^{-\gamma}$, respectively. Thus, $\lambda S(x)$ might be a reasonable upper bound for the number of other strong vertices a strong vertex $x$ can typically infect, either directly or indirectly, before it totally recovers (namely $m_t(x)=0$).

Using this function, Inequality~\eqref{int_condS2} becomes
$$ \frac \lambda {x^{\gamma-1}+ \rho \lambda x^{-\gamma}} \int_0^1 p(x,y) S(y \vee a) dy \le D.$$
We denote the left hand side by $I(x,a)$, and observe that the hypothesis $\gamma>1/2$ implies
$$p(x,y)\le x^{\gamma-1} y^{-\gamma}+ x^{-\gamma} y^{\gamma-1} \le 2p(x,y).$$
If $x\ge a$, we can bound $I$ using the notation $\alpha=\gamma(1-2\eta)$ as 
\begin{align*}
I(x& ,a)= \frac {\lambda}{x^{\gamma-1}+ \rho \lambda x^{-\gamma}} 
\left[\int_0^a p(x,y) S(a) dy +\int_a^1 p(x,y) T_\lambda(y) (y^{\gamma-1}+ \rho \lambda y^{-\gamma}) \,dy \right] \\
&\le \frac {\lambda}{x^{\gamma-1}+ \rho \lambda x^{-\gamma}} 
\left[x^{\gamma-1}S(a)\int_0^a y^{-\gamma} \,dy+ \int_a^1 p(x,y) (1+ \lambda^2 y^{-\alpha}) (y^{\gamma-1}+ \rho \lambda y^{-\gamma}) \,dy \right] \\
&\le \frac {\lambda}{x^{\gamma-1}+ \rho \lambda x^{-\gamma}} 
\Big[ x^{\gamma-1} \Big(\frac {a^{1-\gamma} S(a)}{1-\gamma}+ \int_a^1 (y^{-1}+ \rho \lambda y^{-2\gamma} + \lambda^2 y^{-1-\alpha}+\rho \lambda^3 y^{-\alpha-2 \gamma}) \, dy \Big)\\
& \qquad \qquad+ x^{-\gamma} 
\int_a^1 (y^{2 \gamma-2}+ \rho \lambda y^{-1} + \lambda^2 y^{2\gamma-2-\alpha}+\rho \lambda^3 y^{-1-\alpha}) \,dy\Big] \\
&\le \lambda \Big(\frac {a^{1-\gamma} S(a)}{1-\gamma}+\int_a^1 (y^{-1}+ \rho \lambda y^{-2\gamma} + \lambda^2 y^{-1-\alpha}+\rho \lambda^3 y^{-\alpha-2 \gamma}) \,dy \Big)\\
&\qquad \qquad
+ \frac 1 \rho \int_a^1 (y^{2 \gamma-2}+ \rho \lambda y^{-1} + \lambda^2 y^{2\gamma-2-\alpha}+\rho \lambda^3 y^{-1-\alpha})\, dy,
\end{align*}
At this point, we observe that the bounds we used are tight up to a multiplicative constant. Indeed, replacing a ``max'' by a sum can multiply the result by 2 at worst. The last inequality is tight because taking $x=1$ gives at least $1/(1+\rho \lambda)\ge 1/2$ times the first term (for $\lambda<1/\rho$), while taking $x=a$ gives at least $1/2$ times the second term, if we further suppose
\smash{$\rho \lambda a^{-\gamma}>a^{\gamma-1}$} or 
\smash{$a\le (\rho \lambda)^{\frac 1 {2 \gamma -1}}$} (one can check \emph{a posteriori} that $a_1$ below will always satisfy this property).\smallskip

For simplicity\footnote{We get an additional factor $\log a$ when the exponent is $-1$. However, the reader can check this actually never concerns the leading term, so our results also hold when one of the exponents equals $-1$.}, we now suppose that the exponents in the integrals are different from~$-1$, and use again the inequality \smash{$\int_a^1 y^{\iota} dy \le \frac 1 {|1+\iota|} (a^{1+\iota}+1)$} for $\iota\ne -1$. This allows to give a relatively simple upper bound for $I(x,a)$ (again tight up to a multiplicative constant) as
$$
I(x,a) \lesssim  \frac 1 \rho+\rho \lambda^2+\lambda |\log a|+  \rho \lambda^2 a^{1-2\gamma}+ \lambda^3 a^{-\alpha} + \rho \lambda^4 a^{1-\alpha-2\gamma}+ \frac 1 \rho a^{2\gamma-1}+ \frac 1 \rho \lambda^2 a^{2\gamma-1 - \alpha}.
$$

We want this to be smaller than $D$. To this end, we now fix $\rho=2/D$, so the first term is smaller than $D/2$. Now, we can ensure $I(x,a)$ is smaller than $D$ by requesting
\begin{itemize}
\item $a> r \lambda^{3/\alpha}$ with $r$ large,
\item $a> r \lambda^{2/(2\gamma-1)}$ with $r$ large,
\item $a> r \lambda^{4/(2\gamma+\alpha-1)}$ with $r$ large,
\item $a> r \lambda^{2/(\alpha+1-2 \gamma)}$ with $r$ large (only in the case $\alpha+1-2 \gamma>0$).
\end{itemize}
We then choose $a_1=a_1(\lambda)$ the smallest value making all these requests satisfied. After some more computations, these give 
\begin{itemize}
\item $a_1(\lambda)= r\lambda^{3/\alpha}$ in the case $1/2<\gamma<3/(5+2\eta)$,
\item $a_1(\lambda)= r\lambda^{4/(2\gamma+\alpha-1)}$ in the case $3/(5+2\eta)<\gamma<1/(1+2\eta)$,
\item $a_1(\lambda)= r\lambda^{2/(2\gamma-1)}$ in the case $\gamma>1/(1+2\eta)$.
\end{itemize}
Theorem~\ref{teoupper}, (2), now gives the upper bound for the upper metastable density
$$ \rho^+(\lambda)\le a_1(\lambda) +\frac 1 {S(a_1(\lambda))} \int_{a_1(\lambda)}^1 S(y) \, dy.$$
The reader can check that in all three cases, $S$ is integrable and the integral gives a constant term, while $S(a_1(\lambda))^{-1}$ is of same order as $\lambda a_1^{1-\gamma}$. Actually the expression $\lambda a^{1-\gamma} S(a)$ appears in the upper bound of $I(1,a)$, and the choice of $a_1$ made it small, but of constant order. Finally we get an upper bound of order $\lambda a_1^{1-\gamma}$, which matches~\eqref{dens2}.

\subsubsection*{The upper bound for $\gamma<1/2$ and $\eta<1/2$.}

Here, we define the scoring function~$S$~by
$$ S(x)=(x^{-\gamma}+ \lambda x^{\gamma-1}) T_\lambda(x).$$
and avoid to use the inequality $p(x,y)\le x^{\gamma-1} y^{-\gamma}+ x^{-\gamma} y^{\gamma-1} $, as it is not sharp anymore. Now Inequality~\eqref{int_condS2} is equivalent to the inequality $I(x,a)\le D$ for $x\ge a$, where
\begin{eqnarray*}
 I(x,a)&:=&\frac \lambda { x^{-\gamma}+\lambda x^{\gamma-1}} \left(\int_0^x y^{-\gamma} x^{\gamma-1} S(y \vee a) \, dy + \int_x^1 x^{-\gamma}y^{\gamma-1} S(y) \, dy\right)\\
 &\le& \frac {\lambda x^{\gamma-1}} { x^{-\gamma}+\lambda x^{\gamma-1}} 
 \Big( \frac {a^{1-\gamma}S(a)}{1-\gamma} + \int_a^x (y^{-2\gamma}+\lambda y^{-1}+\lambda^2 y^{-2\gamma-\alpha}+\lambda^3 y^{-1-\alpha}) \, dy\Big) \\
 &&\qquad \qquad + \frac {\lambda x^{-\gamma}} { x^{-\gamma}+\lambda x^{\gamma-1}} 
 \int_x^1 (y^{-1}+\lambda y^{2\gamma -2} + \lambda^2 y^{-1-\alpha} +\lambda^3 y^{2\gamma-2-\alpha} ) \, dy\\
 &\le& I_1(x,a)+I_2(x).
\end{eqnarray*}
We again suppose for simplicity that the exponents do not equal $-1$, and after tedious but straightforward calculations we obtain the following simple upper bounds for $x\in[a,1]$,
\begin{eqnarray*}
I_1(x,a) &\lesssim& \lambda |\log a|+ \lambda^3 a^{-\alpha}+\lambda^2 a^{1-2\gamma - \alpha}, \\
I_2(x) ) &\lesssim& \lambda |\log x|+ \lambda^3 x^{-\alpha}.
\end{eqnarray*}
For example, one of the terms of $I_2(x)$ is 
\begin{eqnarray*}
 \frac {\lambda x^{-\gamma}} { x^{-\gamma}+\lambda x^{\gamma-1}}
  \int_x^1 \lambda^3 y^{2\gamma-2-\alpha}  dy
  &\lesssim& 
   \frac {\lambda x^{-\gamma}} { x^{-\gamma}+\lambda x^{\gamma-1}} \left(\lambda^3 + \lambda^3 x^{2\gamma-1-\alpha}\right)\\
   &\lesssim& 
  \lambda^4 \frac { x^{-\gamma}} { x^{-\gamma}+\lambda x^{\gamma-1}} +
    \lambda^3 x^{-\alpha} \frac { \lambda x^{\gamma-1}} { x^{-\gamma}+\lambda x^{\gamma-1}} \\
&\lesssim& \lambda^4+ \lambda^3 x^{-\alpha} \lesssim  \lambda |\log x|+ \lambda^3 x^{-\alpha},
\end{eqnarray*}
and we bound the other terms similarly. As in the case $\gamma>1/2$ we search for the minimal value making $I_1$ and $I_2$ small, and find that in the region $\gamma<1/2$, $\eta<1/2$, we can always take $a_1(\lambda)= r \lambda^{3/\alpha}$.
Finally, as in the case $\gamma>1/2$, we obtain the upper bound
$$ \rho^+(\lambda)\le a_1(\lambda)+ \frac 1 {S(a_1(\lambda))}\int_{a_1(\lambda)}^1 S(y) \, dy.$$
Again, the integral is of constant order, while $S(a_1(\lambda))^{-1}$ is of same order as $\lambda a_1(\lambda)^{1-\gamma}$, yielding an upper bound matching~\eqref{dens2}.

\subsubsection*{The upper bound for $\gamma<1/2$ and $\eta>1/2$.}

In this case, none of the scoring functions introduced before enable us to prove slow extinction. Besides, the time-scale function is simply $T_\lambda(x)=1$. We define the scoring function $S(x)=x^{-\gamma'}$, with $\gamma'\in(\gamma, 1-\gamma)$. Inequality~\eqref{int_condS2} for $a_1=0$ becomes
$$ \lambda\int_0^1 p(x,y) y^{-\gamma'} dy \le D x^{-\gamma'}.$$
But simple calculations give
\begin{eqnarray*}
\lambda\int_0^1 p(x,y) y^{-\gamma'} dy &=&\lambda x^{\gamma-1} \int_0^x y^{-\gamma-\gamma'} dy + \lambda x^{-\gamma} \int_x^1 y^{\gamma-1-\gamma'} dy\\
&\le& \lambda x^{-\gamma'}\Big(\frac 1{1-\gamma-\gamma'} + \frac 1 {\gamma'-\gamma}\Big),
\end{eqnarray*}
thus Inequality~\eqref{int_condS2} is satisfied for small $\lambda$. Moreover $T_\lambda(x)=1$ satisfies~\eqref{defiT} for every $\delta>0$, thus we have fast extinction, and the expected extinction time grows subpolynomially.
This completes our analysis for the preferential attachment kernel.

\addtocontents{toc}{\setcounter{tocdepth}{2}}
\section{Conclusions}

We have investigated the effect of fast network dynamics on the behaviour of the contact process on scale-free networks modelled as inhomogeneous random network with suitable connection kernels. The stationary network dynamics consists of vertices updating their neighbourhoods independently of the contact process. Variation of a parameter~$\eta$, which controls the rate at which the most powerful vertices update, allows an interpolation between a scenario where vertices update on the time-scale of the contact process ($\eta=0$) and a mean-field model where updates occur on a time scale of much faster order ($\eta\uparrow\infty$). 
We develop general techniques to study the behaviour of the extinction time and metastable densities at small infection rates for this class of models. Lower bounds are based on the identification of four core survival strategies for the contact process, and upper bounds are using coupling and supermartingale techniques. 
\smallskip

Our focus is on two paradigmatic connection kernels, the factor kernel and the preferential attachment kernel, which exhibit very different behaviour.  For the factor kernel we identify a phase transition between fast and slow extinction, and, in case $\eta<\frac12$, a further transition within the slow extinction phase between two types of metastable densities. For the preferential attachment kernel a phase transition between fast and slow extinction only occurs when $\eta>\frac12$. For $\eta<\frac12$ we always have slow extinction and two phase transitions in the behaviour of the metastable densities. The analytical work necessary in the case of the preferential attachment kernel is particularly delicate.
\smallskip

In a future paper we will discuss slowly evolving networks. This will include updating edges individually as well as vertex updating in the case where the most powerful vertices update very slowly ($\eta<0$) and will allow us to interpolate between a scenario where vertices update on the time-scale of the contact process ($\eta=0$) and the case of static networks ($\eta \downarrow -\infty$). The mathematical problems emerging in this work will require, in part, significantly different methods from those explained in this paper.
\pagebreak[3]
\bigskip

\pagebreak[3]

\noindent {\bf Acknowledgements:} \emph{EJ was supported by CNRS and by the LABEX MILYON (ANR-10-LABX-0070) of Universit\'e de Lyon, PM was supported by EPSRC grant EP/K016075/1 and AL was supported by the CONICYT-PCHA/Doctorado nacional/2014-21141160 scholarship. Part of this work was completed when AL was visiting PM at the University of Bath for six months in 2016/17. We would like to thank Remco van der Hofstad for the suggestion to investigate degree dependent update rates.}

\pagebreak[3]

\bigskip


{\footnotesize
\noindent
{\bf Emmanuel Jacob}, Ecole Normale Sup\'erieure de Lyon, Unit\'e de Math\'ematiques Pures et Appliqu\'ees,  UMR CNRS 5669 , 46, All\'ee d'Italie, 69364 Lyon Cedex 07, France

\noindent
{\bf Amitai Linker}, Universidad de Chile,  Department of Mathematical Engineering, Av.~Beauchef 851, piso~5, Santiago, Chile.

\noindent
{\bf Peter M\"orters}, Universit\"at zu K\"oln,  Mathematisches Institut, Weyertal 86--90, 50931~K\"oln, Germany.}

\end{document}